\documentclass[onecolumn,final,sort&compress,a4paper]{elsart3p}

\def\ElsevierStyle

\usepackage{float}
\usepackage{comment}
\usepackage{amsmath}       
\usepackage{amssymb}       

\usepackage{chapterbib}    
\usepackage{color}
\usepackage{comment}
\usepackage{bm,bbm}
\usepackage{overpic}

\usepackage{times}

\usepackage{epsfig}
\usepackage{graphicx,graphics,rotating}
\usepackage{subfigure}
\usepackage{multicol}
\usepackage{multirow}

\usepackage{booktabs,siunitx}

\newcommand{\Tau}{\mathcal{T}}
\newcommand{\Eh} {\mathcal{E}_h}

\newcommand{\vrtx}{\mathsf{v}}

\newcommand{\normV} [1]{\left|\!\left|#1\right|\!\right|_{\V}}

\newcommand{\fh}{f_h}

\usepackage{color}
\definecolor{MyGray}{rgb}{0.61,0.61,0.61}
\definecolor{MyDarkGreen}{rgb}{0,0.45,0}

\newtheorem{theorem}{Theorem}[section]

\newtheorem{remark}[theorem]{Remark}

\newcommand{\BLUE}[1]{{\color{blue}#1}}
\newcommand{\RED}[1]{{\color{red}#1}} 
\def\trait #1 #2 #3 {\vrule width #1pt height #2pt depth #3pt}
\def\fin{\hfill
        \trait .3 5 0
        \trait 5 .3 0
        \kern-5pt
        \trait 5 5 -4.7
        \trait 0.3 5 0
\medskip}

\newenvironment{proof}{\textit{Proof.}}{\fin}

\def\P{{\mathbb P}}

\newcommand{\Vhrs}[1]{{(V^{p_2,p_1}_{h,#1})}^*}

\newcommand{\Vhrp} [1]{V^{p_2,p_1}_{h,#1}}
\newcommand{\VhPrp}[1]{V^{p_2,p_1}_{h,#1}(\P)}

\newcommand{\INTP}{\footnotesize{I}}

\newcommand{\REAL}{\mathbbm{R}}

\newcommand{\TERM}[1]{\textbf{(#1)}}
\newcommand{\TABROW}{}

\newcommand{\EOD}{\end{document}}

\renewcommand{\cv}{\mathbf{c}}

\newcommand{\xv}{\mathbf{x}}
\newcommand{\yv}{\mathbf{y}}

\newcommand{\as}{a}

\newcommand{\es}{e}
\newcommand{\fs}{f}

\newcommand{\qs}{q}

\newcommand{\us}{u}
\newcommand{\vs}{v}

\newcommand{\xs}{x}
\newcommand{\ys}{y}

\newcommand{\Cs}{C}
\newcommand{\Ds}{D}

\newcommand{\Ms}{M}

\newcommand{\Vs}{V}

\newcommand{\Vsp}{V^{\prime}}

\newcommand{\matA}{\mathsf{A}}

\newcommand{\matD}{\mathsf{D}}

\newcommand{\matI}{\mathsf{I}}

\newcommand{\matM}{\mathsf{M}}

\newcommand{\matQ}{\mathsf{Q}}

\newcommand{\matS}{\mathsf{S}}

\newcommand{\matU}{\mathsf{U}}

\newcommand{\LTWO}  {L^2}
\newcommand{\LINF}  {L^{\infty}}

\newcommand{\LS}[1] {L^{#1}}
\newcommand{\HS}[1] {H^{#1}}

\newcommand{\CS}[1] {C^{#1}}

\newcommand{\PS}[1] {\mathbbm{P}_{#1}}

\renewcommand{\P} {\textsf{P}}            
\newcommand  {\E} {e}
\newcommand  {\V} {\mathsf{v}}

\newcommand{\hh}{h}
\newcommand{\Th}{\Omega_{\hh}}

\newcommand{\hP}{\hh_{\P}}

\newcommand{\hE}{\hh_{\E}}
\newcommand{\hV}{\hh_{\V}}

\newcommand{\mP}{\ABS{\P}}

\newcommand{\Pset}{\mathcal{P}}    

\newcommand{\NMB}{N}
\newcommand{\NP}{\NMB^{\Pset}}   

\newcommand{\dV}{\,dV}
\newcommand{\dS}{\,ds}
\newcommand{\dx}{\,d\xv}

\newcommand{\fsh}{\fs_{\hh}}

\newcommand{\qsh}{\qs_{\hh}}

\newcommand{\ush}{\us_{\hh}}

\newcommand{\usI}{\us^{\INTP}}

\newcommand{\vsh}{\vs_{\hh}}

\newcommand{\asP}{\as_{p_1}^{\P}}

\newcommand{\ash}{\as_{\hh}}

\newcommand{\ashP}{\as_{\hh,\P}}

\newcommand{\bil} [2]{\big<#1,#2\big>}

\newcommand{\SP} {S^{\P}}

\newcommand{\nlen}{\hspace{-0.2mm}}

\newcommand{\snorm}  [2]{|#1|_{#2}}

\newcommand{\norm}   [2]{|\nlen|#1|\nlen|_{#2}}

\newcommand{\ABS}    [1]{\left|#1\right|}

\newcommand{\PiPr}[1]{\Pi^{p_1,\P}_{#1}} 
\newcommand{\Pizr}[1]{\Pi^{0,\P}_{#1}}  

\newcommand{\restrict}[2]{{#1}_{|{#2}}}

\newcommand{\Ndofs}{\NMB^{\textrm{dofs}}}

\newcommand{\DOFS}{\textsf{dofs}}

\begin{document}

\begin{frontmatter}
  
  \title{A review on arbitrarily regular conforming virtual element
    methods for elliptic partial differential equations}
  
  \author[MOX]  {P.~F.~Antonietti}
  \author[IMATI]{G.~Manzini}
  \author[UNIMI]{S.~Scacchi}
  \author[MOX]  {M.~Verani}

  \address[MOX]{MOX, Dipartimento di Matematica, Politecnico di
    Milano, Italy}
  
  \address[UNIMI]{Dipartimento di Matematica, Universit\`a degli Studi
    di Milano, Italy}

  \address[IMATI]{IMATI, Consiglio Nazionale delle Ricerche, Pavia,
    Italy }


  \begin{abstract}
    The Virtual Element Method is well suited to the formulation of
    arbitrarily regular Galerkin approximations of elliptic partial
    differential equations of order $2p_1$, for any integer $p_1\geq
    1$.
    In fact, the virtual element paradigm provides a very effective
    design framework for conforming, finite dimensional subspaces of
    $\HS{p_2}(\Omega)$, $\Omega$ being the computational domain and
    $p_2\geq p_1$ another suitable integer number.
    In this study, we first present an abstract setting for such
    highly regular approximations and discuss the mathematical details
    of how we can build conforming approximation spaces with a global
    high-order continuity on $\Omega$.
    Then, we illustrate specific examples in the case of second- and
    fourth-order partial differential equations, that correspond to
    the cases $p_1=1$ and $2$, respectively.
    Finally, we investigate numerically the effect on the
    approximation properties of the conforming highly-regular method
    that results from different choices of the degree of continuity of
    the underlying virtual element spaces and how different
    stabilization strategies may impact on convergence.
  \end{abstract}


\end{frontmatter}


\section{Introduction}
\label{sec1:intro}
In the recent years, there has been an intensive research on numerical
approximations of partial differential equations (PDEs) that can work
on unstructured polygonal and polyhedral (polytopal, for short)
meshes.
Such research activity has led to the design of several families of
numerical discretizations for PDEs, as, for example,
the polygonal/polyhedral finite element
method~\cite{Sukumar-Tabarraei:2004};
the mimetic finite difference
method~\cite{BeiraodaVeiga-Lipnikov-Manzini:2014};
the virtual element method
(VEM)~\cite{BeiraodaVeiga-Brezzi-Cangiani-Manzini-Marini-Russo:2013};
the discontinuous Galerkin method on polygonal/polyhedral
grids~\cite{Antonietti-Giani-Houston:2013,Cangiani-Dong-Georgoulis-Houston:2017};
the hybrid discontinuous Galerkin
method~\cite{Cockburn-Dong-Guzman:2008}; and the
hybrid high--order method~\cite{DiPietro-Droniou:2020}.
Roughly speaking, all these methods are Galerkin-type projection
methods where the solution of a PDE is approximated in a finite
dimensional space that is built upon an underlying mesh made of
arbitrarily-shaped polytopal elements.
In this sense, all such methods can be considered as a generalization
of the finite element method that is formulated on classical
simplicial and quadrilateral meshes.

In particular, the virtual element method, which is the focus of our
paper, has been proven to be very successful in numerical modeling of
scientific and engineering applications.
The conforming VEM was first developed for second-order elliptic
problems in primal
formulation~\cite{BeiraodaVeiga-Brezzi-Cangiani-Manzini-Marini-Russo:2013,BeiraodaVeiga-Brezzi-Marini-Russo:2016b},
and then in mixed
formulation~\cite{BeiraodaVeiga-Brezzi-Marini-Russo:2016c,Brezzi-Falk-Marini:2014}
and nonconforming
formulation~\cite{AyusodeDios-Lipnikov-Manzini:2016}.
A non-exhaustive list of applications
includes
the numerical approximation of
underground flows and discrete fracture
networks~\cite{Benedetto-Berrone-Borio-Pieraccini-Scialo:2014,Benedetto-Berrone-Borio:2016};
propagation and scattering of time-harmonic
waves~\cite{Perugia-Pietra-Russo:2016,Mascotto-Perugia-Pichler:2019a};
topology optimization
problems~\cite{Antonietti-Bruggi-Scacchi-Verani:2017,Chi-Pereira-Menezes-Paulino:2020};
contact mechanics and elasto-plastic deformation
problems~\cite{Wriggers-Rust-Reddy:2016};
phase-field models of isotropic brittle
fractures~\cite{Aldakheel-Hudobivnik-Hussein-Wriggers:2018};
the Schrodinger
equation~\cite{Certik-Gardini-Manzini-Vacca:2018:ApplMath:journal};
obstacle~\cite{Wang-Wei:2018} and minimal surface
problems~\cite{Antonietti-Bertoluzza-Prada-Verani:2020};
nonlocal reaction–diffusion systems describing the cardiac electric
field~\cite{Anaya-Bendahmane-Mora-Sepulveda:2019};
cracks in materials~\cite{Benvenuti-Chiozzi-Manzini-Sukumar:2019};
structural mechanics
problems~\cite{ArtiolideMirandaLovadinaPatruno:2017,Artioli-deMiranda-Lovadina-Patruno:2018};
elastic wave propagation
phenomena~\cite{Park-Chi-Paulino:2019,Park-Chi-Paulino:2020,Antonietti-Manzini-Mazzieri-Mourad-Verani:2021}.
The major reason of this success is that the VEM offers a great
flexibility in designing approximation spaces featuring important
properties other than just supporting polytopal meshes.
Indeed, the VEM features great flexibility in dealing with internal
constraints (e.g., locking phenomena) and in designing ad-hoc
approximation spaces that preserve fundamental properties of the
underlying physical and mathematical models (e.g., incompressibility
constraint).
It is worth mentioning the construction of virtual element spaces
forming de Rham complexes for the Stokes
equations~\cite{BeiraodaVeiga-Lovadina-Vacca:2016}, the Navier-Stokes
equations~\cite{BeiraodaVeiga-Lovadina-Vacca:2018} and the Maxwell
equations~\cite{BeiraodaVeiga-Brezzi-Marini-Russo:2016a}, where the
numerical approximation of the velocity field or the magnetic flux
field is pointwise divergence free as a consequence of the de Rham
inequality chain.
Another remarkable example is provided by the VEM for Helmholtz
problems~\cite{Mascotto-Perugia-Pichler:2019} based on non-conforming
approximation spaces of Trefftz functions, i.e., functions that belong
to the kernel of the Helmholtz operator.

\medskip
In this work, we are interested in the construction of virtual element
spaces with global arbitrarily high smoothness (regularity).
We review the related literature in more details in the next section
as it is the central topic of the present study.
High regularity of the numerical approximation is of primary
importance when dealing with high-order differential problems, i.e.,
problems involving partial differential equations of order $2p_1$,
$p_1\geq 1$, and offers clear advantages even for $p_1=1$, i.e., in
the context of second-order differential equations.
Indeed, global smoothness can be useful to post-process physical
quantities (such as fluxes, strains, stresses), to build exact
discrete Stokes complexes, and to develop anisotropic error estimators
based on the Hessian.
More precisely, the virtual element framework allows us to design
finite dimensional subspaces of $\HS{p_2}(\Omega)$ for some suitable
integer number $p_2\geq1$.
Here, the integer $p_2$ determines the global regularity of the
virtual element functions defined on the computational domain
$\Omega$.
Indeed, the Sobolev Embedding Theorem~\cite{Adams-Fournier:2003}
implies that such functions also belong to $\CS{p_2-1}(\Omega)$ if
$\Omega$ is a bounded, open subset of $\REAL^2$ with a Lipschitz
boundary $\Gamma$ (or with the boundary $\Gamma$ satisfying the ``cone
condition'').
The value of $p_2$ obviously depends on the problem and the numerical
approximation at hand and we will always assume that $p_2\geq p_1$.

\medskip
In the "\emph{classical}" conforming Finite Element Method (FEM), the finite
dimensional spaces are typically only
$\CS{0}$-continuous~\cite{Ciarlet:2002}, and the definition of more
regular approximation spaces is usually considered a difficult task
from both the theoretical and computational viewpoints.
The major difficulty in the formulation of a $\CS{1}$-regular FEM
relies in the explicit construction of a set of basis functions with
such global
regularity~\cite{Argyris-Fried-Scharpf:1968,Bell:1969,Clough-Tocher:1965}.
The remarkable aspect that makes the VEM so appealing in this respect
is that the formulation of such arbitrary regular approximations and
their implementation are relatively straightforward.
The crucial point here is that in the virtual element setting we do
not need to know explicitly the shape functions spanning the virtual
element space.
All the virtual element functions are indeed \emph{virtual} in the
sense that they are implicitly defined as the solution of a local
partial differential equation inside each mesh element.
Consequently, such functions are not explicitly known, with the
noteworthy exception of some subset of polynomials.
Instead, they are uniquely defined by a set of values dubbed the
\emph{degrees of freedom} and these values are the only knowledge that
we really need to formulate and implement the numerical scheme.
This feature makes the construction of arbitrary regular
approximations for any kind of partial differential equations much
simpler and almost immediate.

\medskip
Our first goal in this study is to provide a comprehensive overview of
the state of the art of highly regular conforming virtual element
approximations of PDEs of order $2p_1$, $p_1\geq 1$.
Our second aim is to investigate the influence of different
stabilization strategies on the performance of highly-regular virtual
element discretizations in terms of the condition number of the
resulting linear system of equations and accuracy of the approximation scheme.
For the numerical validation, we focus on two model problems: the
Poisson equation and the biharmonic equation in two spatial
dimensions.
In the next subsection, we provide an overview of the literature
related to arbitrarily regular VEM discretizations.

\subsection{Background material on arbitrarily regular virtual element formulations}

The first work on a $\CS{1}$-regular conforming VEM addressed the
classical plate bending problem~\cite{Brezzi-Marini:2013}.
In such a work, a $\CS{1}$-regular virtual element method is proposed
and analysed for the numerical discretization of the Kirchhoff–Love
model for thin plates.
The approximation error is theoretically proved to decay in the energy
norm, i.e., the $\HS{2}$ norm, with the optimal rate $r-1$, $r\geq2$,
if the local virtual element spaces contain the space of polynomials
of degree $r$.
Optimal errors estimates in both $H^1$ and $L^2$ norms have been
derived later using duality arguments~\cite{Chinosi-Marini:2016}.
Successively, an arbitrarily regular virtual element approximation was
developed for second-order elliptic problems in two-dimensions by
using similar concepts~\cite{BeiraodaVeiga-Manzini:2014} and then
applied to the design of residual-based a-posteriori error
estimators~\cite{BeiraodaVeiga-Manzini:2015}.
A low-order variant of this method was considered for the
semi-discrete approximation of the two-dimensional nonlinear
Cahn-Hilliard
problem~\cite{Antonietti-BeiraodaVeiga-Scacchi-Verani:2016}.
Such VEM needs only three degrees of freedom per mesh vertex and turns
out to be a new discretization also on triangular grids.
Recently, highly regular virtual element spaces have also been
considered for the numerical resolution of the von {K}\'arm\'an
equation modelling the deformation of very thin
plates~\cite{Lovadina-Mora-Velasquez:2019}.
Here, the model under consideration is a fourth-order system of
nonlinear partial differential equations where the unknowns describe
the transverse displacement and the boundary stresses of the plate.
The resulting conforming formulation is shown to be well-posed through
a Banach fixed-point argument provided that the mesh size is small enough, and
optimal errors bounds are proved when the error is measured in the
$\HS{2}$ norm.
Highly-regular conforming VEMs have been recently proposed and
analyzed for general polyharmonic boundary value
problems~\cite{Antonietti-Manzini-Verani:2019} of the form
$(-\Delta)^{p_1}\us=\fs$, $p_1\geq 1$.
The virtual element space of this method contains polynomials of
degree $r \geq 2p_1-2$, features $\CS{p_1-1}$ global regularity and
guarantees optimal approximation bounds in suitable norms, i.e., with
the above introduced notation it corresponds to the choice $p_2=p_1$.
This approach is an extension of the known virtual element
discretization of second- and forth-order problems since the
approximation spaces for $p_1=p_2=1$ and $p_1=p_2=2$ coincide with the
conforming virtual element spaces for the Poisson
equation~\cite{BeiraodaVeiga-Brezzi-Cangiani-Manzini-Marini-Russo:2013}
and the biharmonic equation~\cite{Brezzi-Marini:2013}, respectively.

\medskip
All the previously mentioned works focus onto two-dimensional
mathematical models.
The first highly regular VEM in the three-dimensional setting
addresses the fourth-order linear elliptic
equation~\cite{BeiraodaVeiga-Dassi-Russo:2020}..
The lowest order case requires a virtual element space locally
including quadratic polynomials, i.e., $r=2$.
The degrees of freedom are the values of the virtual element functions
and their gradients at the mesh vertices.

\medskip
A highly regular virtual element method has also been designed for
solving the eigenvalue problem modelling the two-dimensional plate
vibration problem of Kirchhoff
plates~\cite{Mora-Rivera-Velasquez:2018}.
For the resulting spectral problem, the lowest-order
$\HS{2}(\Omega)$-conforming VEM provides the correct spectral
approximation and optimal-order error estimates are derived for the
approximation of the eigenvalues and the eigenfunctions.
Along the same line, a fourth-order spectral problem derived from the
transmission eigenvalue problem is considered in
Reference~\cite{Mora-Velasquez:2018}.
Its variational formulation is written in
$\HS{2}(\Omega)\times\HS{1}(\Omega)$ and the resulting virtual element
approximation is $\HS{2}(\Omega)\times\HS{1}(\Omega)$-conforming.
Employing the classical approximation theory for compact
non-self-adjoint operators, it is shown that the resulting VEM
provides a correct approximation of the spectrum, and the eigenvalues
and eigenfunctions are approximated with the expected (optimal) rates.
The fourth-order plate buckling eigenvalue problem has recently been
addressed~\cite{Mora-Velasquez:2020}.
Here, a $\CS{1}$-regular virtual element method of arbitrary order
$r\geq2$ is used to approximate the buckling coefficients and modes.
This virtual element space is an extension of the approximation space
introduced in References~\cite{Brezzi-Marini:2013}
and~\cite{Antonietti-BeiraodaVeiga-Scacchi-Verani:2016}.
In view of the Babu\v{s}ka--Osborn abstract spectral approximation
theory ~\cite{Babuska-Osborn:1991}, this VEM provides a correct
approximation of the spectrum.
Optimal-order error estimates for the buckling modes and the buckling
coefficients are derived.

\medskip
Finally, it is worth mentioning that in the context of fourth- or
higher-order problems, alternative strategies based on
\emph{non-conforming} approaches are also viable and have been
addressed in the recent literature.
For example, for the biharmonic problem we find $C^0$
\emph{non-conforming} ~\cite{Zhao-Chen-Zhang:2016} and \emph{fully
non-conforming}~\cite{Antonietti-Manzini-Verani:2018,Zhao-Zhang-Chen-Mao:2018}
virtual element approximations, and for higher-order PDEs in $\REAL^n$
we find non-conforming VEM~\cite{Chen-Huang:2020}.
A unified general framework including the lowest-order conforming
VEM~\cite{Brezzi-Marini:2013} and non-conforming
VEM~\cite{Zhao-Chen-Zhang:2016,Antonietti-Manzini-Verani:2018,Zhao-Zhang-Chen-Mao:2018}
has also been proposed and analyzed for the Kirchhoff plate contact
problem with friction~\cite{Wang-Zhao:2020}.
 
\subsection{Outline of the paper}
The remaining part of the manuscript is organized as follows.
In Section~\ref{sec2:continuous_pbl} we introduce the continuous
problem and its weak formulation.
In Section~\ref{sec3:discrete_pbl} we introduce the virtual element
discretization and recall the main abstract convergence result.
In Section~\ref{sec4:VEM} we present the conforming virtual element
approximation with higher-order continuity and recall the main
theoretical results for polyharmonic
problems~\cite{Antonietti-Manzini-Verani:2019}.
Moreover, employing the ideas of
Reference~\cite{BeiraodaVeiga-Manzini:2014}, we extend to the case
$r\geq p_2$ the construction of lower order spaces as considered in
Reference~\cite{Antonietti-Manzini-Verani:2019}.
Section~\ref{sec5:numerics} is devoted to present numerical
experiments for second- and fourth-order elliptic PDEs.
We assess the convergence properties of the VEM versus the mesh size and
the degree of continuity of the underling virtual element space for
different possible choices of the stabilization.
We also investigate numerically how these choices impact on the
condition number of the resulting linear system of equations.
Finally, in Section~\ref{sec6:conclusions} we draw our conclusions.
\section{The continuous problem}
\label{sec2:continuous_pbl}

In this section, we introduce the model problem under investigation
together with its weak formulation.
Let $\Omega\subset\REAL^2$ be an open, bounded, convex domain with
polygonal boundary $\Gamma$.
For any integer $p_1\geq 1$, we introduce the conforming virtual
element method for the approximation of the following problem:
\begin{subequations}
  \label{eq:poly:pblm:continuous}
  \begin{align}
    (-\Delta)^{p_1}\us &= \fs\phantom{0}\qquad\text{in~}\Omega,\label{eq:poly:pblm:1}\\
    \partial^j_n\us   & = 0\phantom{\fs}\qquad\text{for~}j=0,\ldots,p_1-1\text{~on~}\Gamma,\label{eq:poly:pblm:2}
  \end{align}
\end{subequations}
where $\partial^j_n\us$ is the normal derivative of order $j$ of the
function $\us$ with useful conventional notation that
$\partial^0_n\us=\us$.
Let
\begin{align*}
  \Vs\equiv\HS{p_1}_{0}(\Omega) =
  \big\{\vs\in\HS{p_1}(\Omega):\partial^j_n\vs=0\text{~on~}\Gamma,\,j=0,\ldots,p_1-1\big\}.
\end{align*}
Denoting the duality pairing between $\Vs$ and its dual $\Vsp$ by
$\bil{\cdot}{\cdot}$, the variational formulation of the polyharmonic
problem \eqref{eq:poly:pblm:continuous} reads as:
\emph{Find $\us\in\Vs$ such that}
\begin{align}
  \label{eq:poly:pblm:wp}
  \as_{p_1}(\us,\vs) = \bil{\fs}{\vs} \qquad\forall\vs\in\Vs,
\end{align}
where, for any nonnegative integer $\ell$, the bilinear form is given
by:
\begin{align}
  \as_{p_1}(\us,\vs) = 
  \begin{cases}
    \,\int_{\Omega} \nabla\Delta^\ell\us\cdot\nabla\Delta^\ell\vs\,\dx  & \mbox{for~$p_1=2\ell+1$},\\[1em]
    \,\int_{\Omega} \Delta^\ell\us\,\Delta^\ell\vs\,\dx                 & \mbox{for~$p_1=2\ell$}.
  \end{cases}
\end{align}
Whenever $\fs\in\LTWO(\Omega)$ we have
\begin{align}
  \bil{\fs}{\vs} = (\fs,\vs) = \int_{\Omega}\fs\vs\dV\,\dx
  \label{eq:poly:pblm:p-rhs}
\end{align}
where $(\cdot,\cdot)$ denotes the $\LTWO$-inner product.
The existence and uniqueness of the solution to
\eqref{eq:poly:pblm:wp} follows from the Lax-Milgram Theorem because
of the continuity and coercivity of the bilinear form
$\asP(\cdot,\cdot)$ with respect to
$\|\cdot\|_V=\vert\cdot\vert_{p_1,\Omega}$, which is a norm on
$\HS{p_1}_{0}(\Omega)$.
Moreover, since $\Omega$ is a convex polygon, from
Reference~\cite{Gazzola-Grunau-Sweers:1991} we know that
$\us\in\HS{2p_1-m}(\Omega)\cap\HS{p_1}_{0}(\Omega)$ if
$\fs\in\HS{-m}(\Omega)$, $m\leq p_1$ and it holds that
$\norm{\us}{2p_1-m}\leq\Cs\norm{\fs}{-m}$.
In the following, we denote the coercivity and continuity constants of
$\as_{p_1}(\cdot,\cdot)$ by $\alpha$ and $\Ms$, respectively.

Let $\P$ be a polygonal element and denote by
$a_{p_1}^{\P}(\cdot,\cdot)$ the restriction of $a_{p_1}(\cdot,\cdot)$
to $\P$.
For an odd $p_1$, i.e., $p_1=2\ell+1$, a repeated application of the
integration by parts formula yields
\begin{align} 
  \asP(\us,\vs) 
  =
  & -\int_{\P} \Delta^{p_1}\us\,\vs\,\dx + \int_{\partial\P}\partial_n(\Delta^\ell\us)\,\Delta^\ell\vs\dS\nonumber \\[0.5em]
  &  +\sum_{i=1}^\ell
  \left( 
    \int_{\partial\P}\partial_n(\Delta^{p_1-i}\us)\,\Delta^{i-1}\vs\dS
    -\int_{\partial\P}\Delta^{p_1-i}\us\,\partial_n(\Delta^{i-1}\vs)\dS
  \right),
  \label{eq:poly:intbyparts:odd:p}
\end{align}
while, for an even $p_1$, i.e., $p_1=2\ell$, we have
\begin{align} 
  \asP(\us,\vs) 
  &= \int_{\P}\Delta^{p_1}\us\,\vs\,\dx
  \nonumber\\[0.5em]
  &\phantom{=}
  + \sum_{i=1}^\ell
  \left(
  \int_{\partial\P}  \partial_n(\Delta^{p_1-i}\us)\,\Delta^{i-1}\vs\,\dS
  -\int_{\partial\P} \Delta^{p_1-i}\us\,\partial_n(\Delta^{i-1}\vs)\,\dS
  \right).
  \label{eq:poly:intbyparts:even:p}
\end{align}
The above formulas will be crucial to prove the unisolvence of the degrees of freedom of the virtual element spaces and to show the computability of the elliptic projections (cf. Section~\ref{sec3:discrete_pbl}).

\section{The discrete problem and abstract convergence result}
\label{sec3:discrete_pbl}

In this section we present the discrete counterpart of formulation
\eqref{eq:poly:pblm:wp} and recall the abstract convergence result.
Let $\big\{\Th\big\}_{h}$ be a sequence of decompositions of $\Omega$
where each mesh $\Th$ is a collection of nonoverlapping polygonal
elements $\P$ with boundary $\partial\P$, and let $\Eh$ be the set of
edges $\E$ of $\Th$.
Each mesh is labeled by $\hh$, the diameter of the mesh, defined as
usual by $\hh=\max_{\P\in\Th}\hP$, where
$\hP=\sup_{\xv,\yv\in\P}\vert\xv-\yv\vert$.
We denote the set of vertices in $\Th$ by $\mathcal{V}_h$.
The symbol $\hV$ denotes the average of the diameters of the polygons
sharing the vertex $\vrtx$.
For functions in $\Pi_{\P\in\Th}\HS{p_1}(\P)$, we define the seminorm
$\norm{\vs}{\hh}^2=\sum_{\P\in\Th}a_{p_1}^{\P}(\vs,\vs)$.

\medskip
The formulation of the virtual element method for the approximation of
the solution to the elliptic problem~\eqref{eq:poly:pblm:wp} with
arbitrarily smooth functions only requires three mathematical objects:
\begin{enumerate}
\item for $p_2\geq p_1 \geq 1$ the finite dimensional conforming
  virtual element space $\Vhrp{r}\subset\HS{p_2}_0(\Omega)\subset
  \Vs$;
\item the bilinear form $a_{p_1,h}(\cdot,\cdot)$;
\item the linear functional $\langle\fsh,\cdot\rangle$.
\end{enumerate}
Note that the space $\Vhrp{r}$ is made of globally $\CS{k}$ functions
with $k=p_2-1$ and, endowed with suitable degrees of freedom, will be
employed to solve elliptic problems of order $p_1\leq p_2$.

Using such objects, we formulate the VEM as: \emph{Find
$\ush\in\Vhrp{r}$ such that}
\begin{align}
  a_{p_1,h}(\ush,\vsh) = \bil{f_h}{\vsh}
  \quad\forall \vsh\in\Vhrp{r}.
  \label{eq:poly:VEM}
\end{align}
The well-posedness of~\eqref{eq:poly:VEM}, which implies existence and
uniqueness of the solution $\ush$, is a consequence of the Lax-Milgram
lemma.
An abstract convergence result is available, which depends only on the
following assumptions:
\begin{description}
\item [\textbf{(H1)}] for each $h$ and an assigned integer number
  $r\geq p_2$ we are given:
  \begin{enumerate}
    \medskip
  \item the \emph{global} virtual element space $\Vhrp{r}$ with the
    following properties:
    \begin{description}
    \item[-] $\Vhrp{r}$ is a finite dimensional subspace of
      $H^{p_2}_{0}(\Omega)$ and it is made of $C^{k}$ functions with
      $k=p_2-1$;
    \item[-] its restriction $\VhPrp{r}$ to any element $\P$ of a
      given mesh $\Th$, called the \emph{local} (elemental) virtual
      element space, is a finite dimensional subspace of
      $H^{p_2}(\P)$;
    \item[-] {$\PS{r}(\P) \subset \VhPrp{r}$ where $\PS{r}(\P)$ is the
      space of polynomials of degree up to {$r$} defined on $\P$};
    \end{description}
    
    \medskip
  \item the symmetric and coercive bilinear form
    $\as_{p_1,h}:\Vhrp{r}\times\Vhrp{r}\to\mathbbm{R}$ admitting the
    decomposition
    \begin{align*}
      \as_{p_1,h}(\ush,\vsh) = \sum_{\P\in\Th}a^{\P}_{p_1,h}(\ush,\vsh)
      \quad\forall \ush,\,\vsh\in\Vhrp{r},
    \end{align*}
    where each local summation term $a^{\P}_{p_1,h}(\cdot,\cdot)$ is
    also a symmetric and coercive bilinear form;
    
    \medskip
  \item an element $\fsh$ of the dual space $\Vhrs{r}$ of $\Vhrp{r}$,
    which allows us to define the continuous linear functional
    $\bil{f_h}{\cdot}$.
  \end{enumerate}

  \medskip
\item [\textbf{(H2)}] for each $h$ and {each} mesh element $\P\in\Th$,
  the local symmetric bilinear form ${a^{\P}_{p_1,h}(\cdot, \cdot)}$
  possesses the two following properties:
  \begin{description}
  \item[$(i)$] $r$-\textbf{Consistency}: for every polynomial
    $q\in\PS{r}(\P)$ and virtual element function
    $\vsh\in\Vhrp{r}(\P)$ it holds:
    \begin{align}
      a^{\P}_{p_1,h}(\vsh,q) = a_{p_1}^{\P}(\vsh,q);
      \label{eq:poly:r-consistency}
    \end{align}
  \item[$(ii)$] \textbf{Stability}: there exist two positive constants
    $\alpha_*$, $\alpha^*$ independent of $h$ and $\P$ such that for
    every $\vsh\in\VhPrp{r}$ it holds:
    \begin{align}
      \alpha_*\as_{p_1,h}^{\P}(\vsh,\vsh)
      \leq\as_{p_1,h}^{\P}(\vsh,\vsh)\leq
      \alpha^*\as_{p_1}^{\P}(\vsh,\vsh).
      \label{eq:poly:stability}
    \end{align}
  \end{description}

\end{description}
It is easy to check that ${a_{p_1,h}(\cdot, \cdot)}$ is coercive and
continuous.
Let $\PS{r}(\Th)$ denote the space of piecewise (possibly
discontinuous) polynomials defined over the mesh $\Th$.
The following abstract convergence result holds.

\medskip
\begin{theorem}
  \label{theorem:poly:abstract:energy:norm}
  Let $\us$ be the solution of the variational
  problem~\eqref{eq:poly:pblm:wp}.
  Then, for every virtual element approximation $\usI$ in $\Vhrp{r}$
  and any piecewise polynomial approximation $\us_{\pi}\in\PS{r}(\Th)$
  of $\us$ we have:
  \begin{align}
    \normV{\us-\ush}\leq C
    \Big(
    \normV{\us-\usI} + \norm{\us-\us_{\pi}}{h} + \norm{\fh-\fs}{\Vhrs{r}}
    \Big),
    \label{eq:poly:abstract:energy:norm}
  \end{align}
  where $C$ is a constant independent of $\hh$ that may depend on
  $\alpha$, $\alpha_*$, $\alpha^*$, $M$, {and $r$,} and
  \begin{align}
    \norm{\fs-\fh}{\Vhrs{r}}
    = \sup_{\vsh\in\Vhrp{r}\backslash{\{0\}}}\frac{\bil{\fs-\fh}{\vsh}}{\normV{\vsh}}
  \end{align}
  is the approximation error of the right-hand side given in the norm
  of the dual space $\Vhrs{r}$.
\end{theorem}

\begin{proof}
We report here the proof for
completeness~\cite{Antonietti-Manzini-Verani:2019}.
First, an application of the triangular inequality implies that:
  \begin{align}
    \normV{\us-\ush}\leq\normV{\us-\usI}+\normV{\usI-\ush}.
    \label{eq:poly:abstract:proof:00}
  \end{align}
  Let $\delta_h=\ush-\usI$.
  Starting from the definition of $\normV{\,\cdot\,}$, we find that:
  \begin{align*}
    \begin{array}{rll}
      &\alpha_*\normV{\delta_h}^2
      = \alpha_*a_{p_1}(\delta_h,\delta_h)                                                                                                              &\mbox{\big[use~\eqref{eq:poly:stability}\big]}   \nonumber\\[0.5em]
      &\quad\leq a_{p_1,h}(\delta_h,\delta_h)                                                                                                           &\hspace{-1.2cm}\mbox{\big[use~$\delta_h=\ush-\usI$\big]}         \nonumber\\[0.5em]
      &\quad\leq a_{p_1,h}(\delta_h,\ush)-a_{p_1,h}(\delta_h,\usI)                                                                                       &\mbox{\big[use~\eqref{eq:poly:VEM}\big]}          \nonumber\\[0.5em]
      &\quad\leq \bil{\fh}{\delta_h}-\sum_{\P\in\Th}a^{\P}_{p_1,h}(\delta_h,\usI)                                                                          &\mbox{[add $\pm\us_{\pi}$\big]}                    \nonumber\\[0.5em]
      &\quad\leq \bil{\fh}{\delta_h}-\sum_{\P\in\Th}\Big(a^{\P}_{p_1,h}(\delta_h,\usI-u_{\pi})+a^{\P}_{p_1,h}(\delta_h,u_{\pi})\Big)                            &\mbox{\big[use~\eqref{eq:poly:r-consistency}\big]}\nonumber\\[0.5em]
      &\quad\leq \bil{\fh}{\delta_h}-\sum_{\P\in\Th}\Big(a^{\P}_{p_1,h}(\delta_h,\usI-u_{\pi})+a^{\P}_{p_1}  (\delta_h,u_{\pi})\Big)                            &\mbox{[add $\pm\us$\big]}                         \nonumber\\[0.5em]
      &\quad\leq \bil{\fh}{\delta_h}-\sum_{\P\in\Th}\Big(a^{\P}_{p_1,h}(\delta_h,\usI-u_{\pi})+a^{\P}_{p_1}  (\delta_h,u_{\pi}-u) + a^{\P}_{p_1}(\delta_h,u)\Big) &\mbox{\big[use~\eqref{eq:poly:pblm:wp}\big]}      \nonumber\\[0.5em]
      &\quad=    \bil{\fh-f}{\delta_h}-\sum_{\P\in\Th}\Big(a^{\P}_{p_1,h}(\delta_h,\usI-u_{\pi})+a^{\P}_{p_1}  (\delta_h,u_{\pi}-u)\Big).
    \end{array}
  \end{align*}
  Then, we use~\eqref{eq:poly:stability}, add and subtract $\us$, use
  the continuity of $a^{\P}_{p_1}$, sum over all the elements $\P$,
  divide by $\normV{\delta_h}$, take the supremum of the right-hand
  side error term on $\Vhrp{r}\backslash{\{0\}}$, and obtain
  \begin{align}
     \alpha_*\normV{\delta_h} \leq
     \sup_{ \vsh\in\Vhrp{r}\backslash{ \{0\} } } \frac{\vert\bil{\fh-\fs}{\vsh}\vert}{\normV{\vsh}} 
     + M\left( \alpha^*\normV{\usI-\us} + (1+\alpha^*)\norm{\us-\us_{\pi}}{h} \right).
     \label{eq:poly:abstract:proof:15}
  \end{align}
  The assertion of the theorem follows by
  substituting~\eqref{eq:poly:abstract:proof:15}
  in~\eqref{eq:poly:abstract:proof:00} and suitably defining the
  constant $C$.
\end{proof}
\section{The virtual element spaces of higher-order continuity}
\label{sec4:VEM}

\subsection{Preliminaries}

The ``\emph{degrees of freedom tuples}'' are a very effective way to
characterize the \emph{set of degrees of freedom (\DOFS{})} that
uniquely identify the virtual element functions as members of a finite
dimensional subspace of a $\CS{k}$-regular virtual element space.
Our degrees of freedom tuple, abbreviated as ``\DOFS{}-tuple'', is a
generalization of the similar concept that was originally introduced
for the degrees of freedom of a nonconforming virtual element
space~\cite{Dedner-Hodson:2021}.
Our \DOFS{}-tuple is an array $M_k\in\mathbb{Z}^{2(k+1)+1}$ defined by
\begin{equation}
  \label{Mdofs}
  M_k = \Big(\,\big(d^\V_0,\ldots d^\V_k\big),\,\big(d^{\E}_0,\ldots,d^{\E}_k\big),\,d^i_0\,\Big).
\end{equation}
The integer variables $\big(d^\V_j\big)$ and $\big(d^{\E}_j\big)$, for
$j=0,\ldots,k$, respectively encode the information associated with
the mesh vertices and mesh edges; the last integer variable $d^i_0$
encodes the information associated with the interior of the mesh
elements $\P$.
The subscript $j=0$ in $d^{\V}_0$, $d^{\E}_0$, and $d^i_0$ indicates
that these variables refer to the virtual element function.
The subscript values $j=1,\ldots,k$ in $d^{\V}_j$ and $d^{\E}_j$
denote the reference to the partial derivatives
$\Ds^{\nu}=\partial^{\ABS{\nu}}\slash{\partial\xs^{\nu_1}\partial\ys^{\nu_2}}$
of order $\ABS{\nu}=\nu_1+\nu_2=j$ of the virtual element function
($\nu=(\nu_1,\nu_2)$ being a multi-index).
The vertex variables $d^\V_j$ can only take the values $-1$ or $0$,
while the edge variables $d^{\E}_j$ and the elemental variable $d^i_0$
either take the value $-1$ or a nonnegative integer value.
If the entry is equal to $-1$, the corresponding term is not used as a
degree of freedom.
If $d^{\V}_j=0$, the $j$-th order partial derivatives evaluated at the
mesh vertices are in the set of degrees of freedom (with the usual
convention that $\Ds^{\nu}\vsh(\V)=\vsh(\V)$ $\nu=(0,0)$, i.e.,
$j=0$).
A nonnegative value of $d^{\E}_0$ and $d^i_0$ defines the maximum
order of the polynomial moments used in the definition of the degrees
of freedom associated with the elemental edges $\E\in\partial\P$ and
the interior of the element $\P$.

\medskip
By using the \DOFS{}-tuple $M_k$, we define the following set of
values of a function $\vs\in\HS{k+1}(\P)$:
\begin{description}
\item\TERM{D1} $\hV^{|\nu|}\Ds^{\nu}\vsh(\V)$ at all vertices $\V$ of
  the polygonal boundary $\partial\P$, for every multi-index
  $\nu=(\nu_1,\nu_2)$ such that $\ABS{\nu}=j$ if $d^{\V}_j=0$,
  $j=0,\ldots,k$;

  \medskip  
\item\TERM{D2}
  $\displaystyle\hE^{-1+j}\int_{\E}\qs\partial^j_n\vsh\dS$ for any
  $\qs\in\PS{d^e_j}(\E)$, $j=0,\ldots,k$ and any edge $\E$ of
  $\partial\P$;

  \medskip
\item\TERM{D3} $\displaystyle\hP^{-2}\int_{\P}\qsh\vsh\,\dx$ for
  any $\qs\in\PS{d^i_0}(\P)$.
\end{description}

\subsection{Local and global spaces}

For $p_2\geq p_1\geq 1$ we first consider the case $r\geq 2p_2-1$,
while the lower order case $p_2 \leq r \leq 2p_2 -1$ will be addressed
in Section \ref{S:lower}.
The local virtual element space on element $\P$ is defined by
\begin{multline}\label{eq:vem-space-higher}
  \VhPrp{r} = \Big
  \{\vsh\in\HS{p_2}(\P):\,
  \Delta^{p_2}\vsh\in\PS{r-2p_1}(\P),\,
  \partial^i_n\vsh\in\PS{r-i}(\E),\,\\
  i=0,\ldots,p_2-1~\forall
  \E\in\partial\P
  \Big\},
\end{multline}
with the conventional notation that $\PS{-1}(\P)=\{0\}$.
The virtual element space $\VhPrp{r}$ contains the space of
polynomials $\PS{r}(\P)$, for $r\geq 2p_2-1$.

We take $k=p_2-1$ in \eqref{Mdofs} and endow the local space
$\VhPrp{r}$ with the \DOFS{}-tuple $M_{p_2-1} = M_{p_2-1}(p_1)$, which
depends on the parameter $p_1$ by setting
\begin{align*}
  & d^\V_{j}=0        \quad j=0,\ldots,p_2-1\\[0.5em]
  & d^e_{j}=r-2p_2+j  \quad j=0,\ldots,p_2-1\\[0.5em]
  & d^i_0=r-2p_1.
\end{align*}
Employing
\eqref{eq:poly:intbyparts:odd:p}-\eqref{eq:poly:intbyparts:even:p} it
is possible to prove that the degrees \TERM{D1}-\TERM{D3} defined
through the \DOFS{}-tuple $M_{p_2-1}(p_1)$ are unisolvent in
$\VhPrp{r}$, see Reference~\cite{Antonietti-Manzini-Verani:2019}.
The particular choice of $d^i_0$ is essential for the computability of
the elliptic projection with respect to $\asP(\cdot,\cdot)$, which is
a scalar product in $H^{p_2}_{0}(\Omega)$
cf. Remark~\ref{rem:elliptic_projection}.

Building upon the local spaces $\VhPrp{r}$ for all $\P\in\Th$, the
\emph{global} conforming virtual element space $\Vhrp{r}$ is defined
on $\Omega$ as
\begin{align}
  \Vhrp{r} = \Big\{
  \vsh\in\HS{p}_{{0}}(\Omega)\,:\,\restrict{\vsh}{\P}\in\VhPrp{r}\,\,\forall\P\in\Th
  \Big\}.
  \label{eq:poly:global:space}
\end{align}

The set of global degrees of freedom inherited by the local degrees of
freedom defined by $M_{p_2-1}(p_1)$ are:
\medskip
\begin{itemize}
\item $\hV^{|\nu|}\Ds^{\nu}\vsh(\V)$, $\ABS{\nu}\leq p_2-1$ for every
  interior vertex $\V$ of $\Th$;
  \medskip
\item $\displaystyle\hE^{-1+j}\int_{\E}\qs\partial^j_n\vsh\,ds$ for
  any $\qs\in\PS{r-2p_2+j}(e)$ $j=0,\ldots,p-1$ and every interior
  edge $\E\in\Eh$;
\medskip
\item $\displaystyle\hP^{-2}\int_{\P}\qs\vsh\dx$ for any
  $\qs\in\PS{r-2p_1}(\P)$ and every $\P\in\Th$.
\end{itemize}

We remark that the associated global space is made of
$\HS{p_2}(\Omega)$ functions.
Indeed, the restriction of a virtual element function $\vsh$ to each
element $\P$ belongs to $\HS{p_2}(\P)$ and glues with
$C^{p_2-1}$-regularity across the internal mesh faces.

\begin{remark}[Examples]
  We report some relevant examples from the virtual element literature
  that are included in the above abstract framework:
  \begin{itemize}
  \item for $p_1=p_2=1$, we obtain the $\CS{0}$-conforming virtual
    element space for the Poisson
    equation~\cite{BeiraodaVeiga-Brezzi-Cangiani-Manzini-Marini-Russo:2013};
  \item for $p_2=p_1=2$ we obtain the conforming virtual element space
    for the biharmonic equation~\cite{Brezzi-Marini:2013};
  \item for $p_1=1$ and $p_2=2$ we obtain the $\CS{1}$-conforming
    virtual element space for the Poisson
    equation~\cite{BeiraodaVeiga-Manzini:2014};
  \item For $p_1=1$ and $p_2=3$ we obtain the $\CS{2}$-conforming
    virtual element space for the Poisson
    equation~\cite{BeiraodaVeiga-Manzini:2014}.
  \end{itemize}
\end{remark}

\subsection{Lower-order virtual spaces}
\label{S:lower}
Lower-order elemental spaces~\cite{BeiraodaVeiga-Manzini:2014} can be
defined that contains the subspace of polynomials of degree up to $r$
with $p_2\leq r\leq 2p_2-2$:
\begin{multline}\label{eq:vem-space-lower}
  \VhPrp{r} = \Big\{\vsh\in\HS{p_2}(\P):\, \Delta^{p_2}\vsh \in
  \PS{r-2p_1}(\P),
  \,\partial^i_n\vsh\in\PS{\alpha_i}(\E),\,\\
  \,i=0,\ldots,p_2-1~\forall\E\in\partial\P
  \Big\},
\end{multline}
where $\alpha_j=\max\{2p_2-1-2j, r-j\}$. 

For $r = 2 p_2 -1-k$ with $k=0,1,\ldots,p_2-1$, the virtual element
functions in the elemental space \eqref{eq:vem-space-lower} are
uniquely identified by the degrees of freedom of the \DOFS{}-tuple
${M}_{p_2-1}(p_1)$ by setting
\begin{align*}
  & d^\V_{j}=0 \quad j=0,\ldots,p_2-1,         \\[0.5em]
  & d^\E_{j}=-1\quad j=0,\ldots,k,             \\[0.5em]
  & d^\E_{j}=\alpha_j \quad j=k+1,\ldots,p_2-1,\\[0.5em]
  & d^i_0=r-2p_1.
\end{align*}
Equivalently, 
\begin{description}
\item\TERM{D1} $\hV^{|\nu|}D^{\nu}\vsh(\V)$, $\ABS{\nu}\leq p_2-1$
  for any vertex $\V$ of $\partial\P$;
  
  \medskip
\item\TERM{D2}
  $\displaystyle\hE^{-1+j}\int_{\E}\qs\partial_{n}^j\vsh\dS$ for any
  $\qs\in\PS{\alpha_j}(\E)$ and edge $\E$ of $\partial\P$,
  $j=k+1,\ldots,p_2-1$.

  \medskip
\item\TERM{D3} $\displaystyle\hP^{-2}\int_{\P}\qs\vsh\dx$ for any
  $\qs\in\PS{r-2p_1}(\P)$ and every $\P\in\Th$.
\end{description}

The above set of degrees of freedom is unisolvent in $\VhPrp{r}$ and
allows the computability of the elliptic projection $\PiPr{r}$ with
respect to $\asP(\cdot,\cdot)$.
The global virtual element space $\Vhrp{r}$ is built as in the
previous section and is made of $C^{p_2-1}$ functions.
\begin{remark}
  The virtual space $\VhPrp{r}$ in \eqref{eq:vem-space-lower} for
  $r=2p_2-2$ has been first introduced in the work of
  Reference~\cite{Antonietti-Manzini-Verani:2019}, while the virtual
  element spaces for $p_2\leq r < 2p_2-2$ are new.
  In the lowest order case ($r=p_2$) the local virtual element space $
  \VhPrp{r}$ does not employ the \DOFS{} defined in \TERM{D2}, so the
  corresponding \DOFS{}-tuple is equal to:
  \begin{align*}
    {M}_{p_2-1}(p_1) =
    (0,\ldots,0,-1,\ldots,-1,d^i_0).
  \end{align*}
  In particular, for $p_1=p_2=2$ and $r=2$ we obtain the space
  introduced in
  Reference~\cite{Antonietti-BeiraodaVeiga-Scacchi-Verani:2016}
  for the conforming approximation of the Cahn-Hilliard equation.
  For $p_1=1$ and $r\geq p_2\geq 2$ we obtain the spaces introduced in
  Reference~\cite{BeiraodaVeiga-Manzini:2014} for the virtual element
  approximation of the Laplace problem with arbitrary regularity.
  The space $\VhPrp{p_2}$ with $p_1=1,2$ will be employed in
  Section~\ref{sec5:numerics} to perform numerical tests.
  Finally, we note that $r\geq 2p_2-1$ implies $\alpha_j=r-j$ and
  \eqref{eq:vem-space-lower} reduces to \eqref{eq:vem-space-higher}.
\end{remark}

\subsection{Projection operators and discrete bilinear forms}

The choice of $d_0^i$ in the \DOFS{}-tuple $M_{p_2-1}(p_1)$ is crucial
for the computability of the elliptic projection
$\PiPr{r}:\VhPrp{r}\to\PS{r}(\P)$, {\em with respect to}
$\asP(\cdot,\cdot)$.
This fact will become clear in the discussion below (see Remark
\ref{rem:elliptic_projection}).
To define the elliptic projection we need the \emph{vertex average
projector} $\widehat{\Pi}^{\P}:\VhPrp{r}\to\PS{0}(\P)$, which projects
any (smooth enough) function defined on $\P$ onto the space of
constant polynomials.
Let $\psi$ be a continuous function defined on $\P$.
The \emph{vertex average projection} of $\psi$ onto the constant
polynomial space is given by:
\begin{align}
  \widehat{\Pi}^{\P}\psi = \frac{1}{\NP}\sum_{\vrtx\in\partial\P}\psi({\vrtx}).
  \label{eq:trih:vertex:average:projection}
\end{align}
The elliptic projection $\PiPr{r}:\VhPrp{r}\to\PS{r}(\P)$ is the
solution of the finite dimensional variational problem:
\begin{align}
  \asP(\PiPr{r}\vsh,\qs)                &= \asP(\vsh,\qs)\phantom{ \widehat{\Pi}^{\P}\Ds^{\nu}\vsh }\forall\qs\in\PS{r}(\P),\label{eq:poly:Pi:A}\\[0.5em]
  \widehat{\Pi}^{\P}\Ds^{\nu}\PiPr{r}\vsh &= \widehat{\Pi}^{\P}\Ds^{\nu}\vsh\phantom{\asP(\vsh,\qs)} \ABS{\nu}\leq{p_2-1}.      \label{eq:poly:Pi:B}
\end{align}

Employing \eqref{eq:poly:intbyparts:odd:p}-
\eqref{eq:poly:intbyparts:even:p}, in
Reference~\cite{Antonietti-Manzini-Verani:2019} it is shown that such
operator has two important properties:
\begin{itemize}
\item[$(i)$] it is polynomial-preserving in the sense that
  $\PiPr{r}\qs=\qs$ for every $\qs\in\PS{r}(\P)$;
\item[$(ii)$] the polynomial projection $\PiPr{r}\vsh$ is
  \emph{computable} using only the degrees of freedom of
  $\vsh\in\VhPrp{r}$ that are specified by the \DOFS{}-tuple
  $M_{p_2-1}(p_1)$.
\end{itemize}

\begin{remark}[On the role of $ d_0^i$ in the computability of $\PiPr{r}$]
  \label{rem:elliptic_projection}
  We report a simple, but instructive example to clarify that the
  computability of $\PiPr{r}$ is related to the interplay between the
  parameter $p_1$ (dictating the scalar product employed in the
  definition of the elliptic projection) and the degrees of freedom
  specified by \DOFS{}-tuple $M_{p_2-1}(p_1)$.

  For $p_1=1$ and $p_2=2$, $\vsh\in\VhPrp{r}$ and $\qs\in\PS{r}$, we
  have that
  \begin{eqnarray}
    \asP(\vsh,\qs) &=& \int_{\P}\nabla \vsh\cdot\nabla \qs 
    ~\dx=-\int_\P \vsh \Delta q~\dx + 
    \int_{\partial\P} \vsh\partial_n q ~\dS.\nonumber
  \end{eqnarray}
  As $\Delta\qs\in\PS{r-2}$, the first term in the last equality on
  the right is computable in view of the choice of the degrees of
  freedom \TERM{D3} with $d_0^i=r-2p_1=r-2$.
  The computability of the second term follows from the fact that the
  trace of $\vsh$ on each edge of $\P$ is a polynomial that can be
  computed explicitly by interpolating the values in \TERM{D1} and
  \TERM{D2}.
  
\end{remark}

Now, we introduce the symmetric bilinear form
$\ash:\Vhrp{r}\times\Vhrp{r}\to\REAL$, which is written as the sum of
local terms
\begin{align}
  \as_{p_1,h}(\ush,\vsh) = \sum_{\P\in\Th}a^\P_{p_1,h}(\ush,\vsh),
\end{align}
where each local term $a^\P_{p_1,h}:\VhPrp{r}\times\VhPrp{r}\to\REAL$
is a symmetric bilinear form.
We set
\begin{align}
  \ashP(\ush,\vsh) 
  = \asP(\PiPr{r}\ush,\PiPr{r}\vsh) 
  + \SP(\ush-\PiPr{r}\ush,\vsh-\PiPr{r}\vsh),
  \label{eq:poly:ah:def}
\end{align}
where the stabilization form $\SP:\VhPrp{r}\times\VhPrp{r}\to\mathbbm{R}$ is a symmetric
positive definite bilinear form such that
\begin{align}
  \sigma_*\asP(\vsh,\vsh)\leq\SP(\vsh,\vsh)\leq\sigma^*\asP(\vsh,\vsh)
  \qquad\forall\vsh\in\VhPrp{r}\textrm{~with~}\PiPr{r}\vsh=0,
  \label{eq:poly:S:stability:property}
\end{align}
for two positive constants $\sigma_*$, $\sigma^*$ that are independent
of $\hh$ (and $\P$).
The bilinear form $\ashP(\cdot,\cdot)$ has the two fundamental
properties of $r$-\emph{consistency} and \emph{stability},
cf.~\eqref{eq:poly:r-consistency} and
\eqref{eq:poly:stability}~\cite{Antonietti-Manzini-Verani:2019}.

\subsection{Discretization of the load term}
Let $\fsh$ be the piecewise polynomial approximation of $\fs$ on $\Th$
given by
\begin{equation}
  \label{eq:rhs}
  \restrict{\fsh}{\P} = \Pizr{r-p_1}\fs,
\end{equation}
for $r\geq p_2$ and $\P\in\Th$.
Then, we set
\begin{equation}\label{vem:rhs}
  \bil{\fsh}{\vsh} = \sum_{\P\in\Th}\int_{\P}\fsh\vsh\,{dx dy}.
\end{equation}
Using the definition of the $\LTWO$-orthogonal projection we find that
\begin{equation}\label{aux:1.1}
  \bil{\fsh}{\vsh} 
  = \sum_{\P\in\Tau_h}\int_{\P}\Pizr{r-p_1}\fs\,\vsh\,{dx dy}
  = \sum_{\P\in\Tau_h}\int_{\P}\Pizr{r-p_1}\,\fs\Pizr{r-p_1}\vsh\,{dx dy}
  = \sum_{\P\in\Tau_h}\int_{\P}f\,\Pizr{r-p_1}\vsh\,{dx dy}. 
\end{equation}
The right-hand side of \eqref{aux:1.1} is computable by using the
degrees of freedom \TERM{D1}-\TERM{D3} and the enhanced
approach~\cite{Ahmad-Alsaedi-Brezzi-Marini-Russo:2013}.

\subsection {Error analysis}
In this section, we recall some convergence results for the
approximation of \eqref{eq:poly:pblm:continuous}.
In particular, employing Theorem
\ref{theorem:poly:abstract:energy:norm} together with standard
approximation results and assuming the use of the enhanced
spaces~\cite{Ahmad-Alsaedi-Brezzi-Marini-Russo:2013} to provide
optimal approximation properties of the right hand side, the following
convergence result in the energy norm
holds~\cite{Antonietti-Manzini-Verani:2019}

\begin{theorem}
  \label{theorem:poly:energy:convg:rate}
  Let $\fs\in\HS{r-p_1+1}(\Omega)$ be the forcing term at the
  right-hand side, $\us$ the solution of the variational
  problem~\eqref{eq:poly:pblm:wp} and $\ush\in\Vhrp{r}$ the solution
  of the virtual element method~\eqref{eq:poly:VEM}.
  Then, it holds that
  \begin{align}
    \normV{\us-\ush}
    \leq\Cs\hh^{r-(p_1-1)}\big( \snorm{\us}{r+1} + \snorm{\fs}{r-p_1+1} \big).
    \label{eq:poly:energy:convg:rate}
  \end{align}
\end{theorem}

Moreover, the following convergence results in lower order norms can
be established~\cite{Antonietti-Manzini-Verani:2019}.
\begin{theorem}[Even $p_1$, even norms]
  Let $\fs\in\HS{r-p_1+1}(\Omega)$, $\us$ the solution of the
  variational problem~\eqref{eq:poly:pblm:wp} with $p_1=2\ell$ and
  $\vsh\in\Vhrp{r}$ the solution of the virtual element
  method~\eqref{eq:poly:VEM}.
  Then, there exists a positive constant $\Cs$ independent of $\hh$
  such that
  \begin{align}
    \snorm{\us-\ush}{2i}
    \leq\Cs\hh^{r+1-2i}\Big(\snorm{\us}{r+1}+\snorm{\fs}{r-(p_1-1)}\Big),
  \end{align}
  for every integer $i=0,\ldots,\ell-1$.
\end{theorem}
\begin{theorem}[Even $p_1$, odd norms]
  Let $\fs\in\HS{r-p_1+1}(\Omega)$, and $\us$ the solution of the
  variational problem~\eqref{eq:poly:pblm:wp} with $p_1=2\ell$ and
  $\ush\in\Vhrp{r}$ the solution of the virtual element
  method~\eqref{eq:poly:VEM}.
  Then, there exists a positive constant $\Cs$ independent of $\hh$
  such that
  \begin{align}
    \snorm{\us-\ush}{2i+1}
    \leq\Cs\hh^{(r+1)-(2i+1)}\Big(\snorm{\us}{r+1}+\snorm{\fs}{r-(p_1-1)}\Big),
  \end{align}
  for every integer $i=0,\ldots,\ell-1$.
\end{theorem}
\begin{theorem}[Odd $p_1$, even norms]
  Let $\fs\in\HS{r-p_1+1}(\Omega)$ and $\us$ be the solution of the
  variational problem~\eqref{eq:poly:pblm:wp} and $\ush\in\Vhrp{r}$
  the solution of the virtual element method~\eqref{eq:poly:VEM}.
  Then, there exists a positive constant $\Cs$ independent of $\hh$
  such that
  \begin{align}
    \snorm{\us-\ush}{2i}
    \leq\Cs\hh^{ (r+1)-2i }\Big(\snorm{\us}{r+1} + \snorm{\fs}{r-(p_1-1)}\Big),
  \end{align}
    for every integer $i=0,\ldots,\ell-1$.
\end{theorem}
\begin{theorem}[Odd $p_1$, odd norms]
  Let $\fs\in\HS{r-p_1+1}(\Omega)$ and $\us$ be the solution of the
  variational problem~\eqref{eq:poly:pblm:wp} and $\ush\in\Vhrp{r}$
  the solution of the virtual element method~\eqref{eq:poly:VEM}.
  Then, there exists a positive constant $\Cs$ independent of $\hh$
  such that
  \begin{align}
    \snorm{\us-\ush}{2i+1}
    \leq\Cs\hh^{(r+1)-(2i+1)}\Big(\snorm{\us}{r+1}+\snorm{\fs}{r-(p_1-1)}\Big),
  \end{align}
  for every integer $i=0,\ldots,\ell-1$.
\end{theorem}
\newcommand{\stabI}{$\matU=\matI$}
\newcommand{\stabD}{$\matU=\matD^{\perp}$}
\newcommand{\alphaOne}  {\multicolumn{2}{c|}{$\alpha_{\mathsf{stab}}=\text{Trace}(\matM_{\P})/3$}}
\newcommand{\alphaTwo}  {\multicolumn{2}{c}{$\alpha_{\mathsf{stab}}=1\slash{\mP}$             }}
\newcommand{\alphaThree}{\multicolumn{2}{c}{$\alpha_{\mathsf{stab}}=1\slash{\hh^2}$           }}

\section{Numerics}
\label{sec5:numerics}
We investigate the behavior of the two-dimensional, highly-regular,
conforming virtual element approximations that we introduced in the
previous sections when applied to the numerical resolution of the
Poisson ($p_1=1$) and biharmonic ($p_1=2$) equations.

According to the notation introduced in
Section~\ref{sec3:discrete_pbl}, we recall that the finite dimensional
virtual element space $\Vhrp{r} \subset H^{p_2}_{0}(\Omega)$ is made
of $C^{k}$-regular functions on $\Omega$ where $k=p_2-1$.
Moreover, the local Virtual Element space $\VhPrp{r}$, i.e. the
restriction of $\Vhrp{r}$ to any element $\P\in\Th$, is a finite
dimensional subspace of $H^{p_2}(\P)$ containing the space of
polynomials of degree up to {$r$} defined on $\P$.

\medskip
Throughout the section, the computational domain is the unit square,
the loading term $\fs$ is set up in accordance with the exact solution
\begin{align*}
  \us(\xs,\ys)=(1-\xs)^2\xs^2\,(1-\ys)^2\ys^2,
\end{align*}
and the boundary conditions are chosen accordingly.

\medskip
We consider four different mesh families: quadrilateral meshes
\textsf{QUAD}, triangular meshes \textsf{TRI}, central Voronoi
tessellations \textsf{CVT} and hexagonal meshes \textsf{HEX}.
An example of a mesh of each family is shown in Fig.~\ref{fig_mesh};
the corresponding number of elements of the refined meshes is shown in
Table~\ref{tab:meshes:elements}.
\begin{figure}[!t]
  \begin{center}
    \def\arraystretch{1}\tabcolsep=5pt
    \begin{tabular}{cc}
      \subfigure[Quadrilateral (\textsf{QUAD}) mesh]{\includegraphics[scale=0.125]{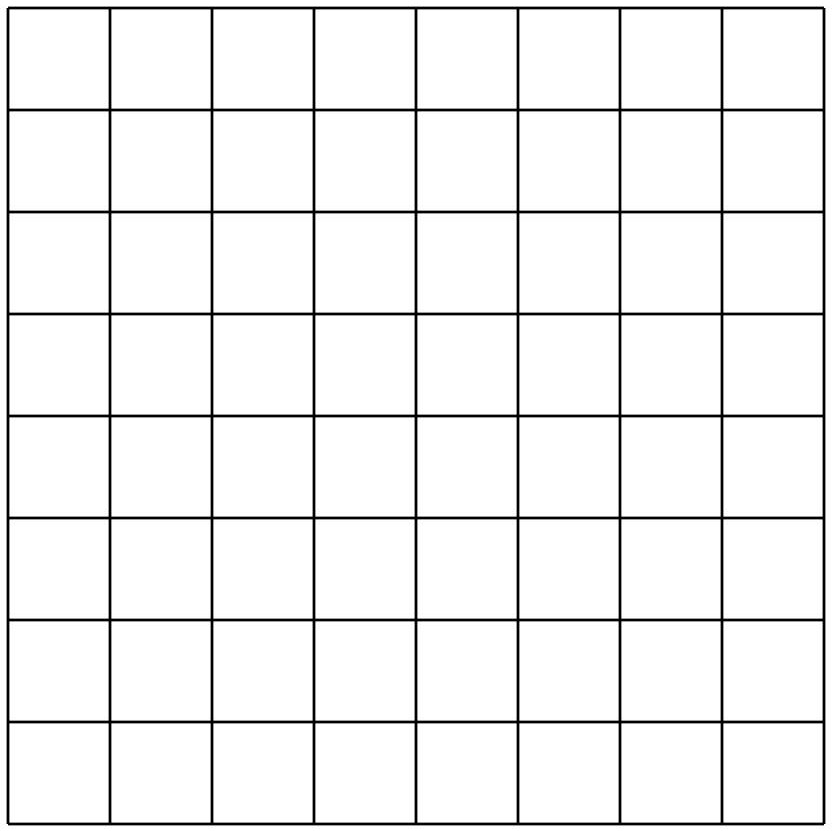}} &\qquad  
      \subfigure[Triangular    (\textsf{TRI})  mesh]{\includegraphics[scale=0.125]{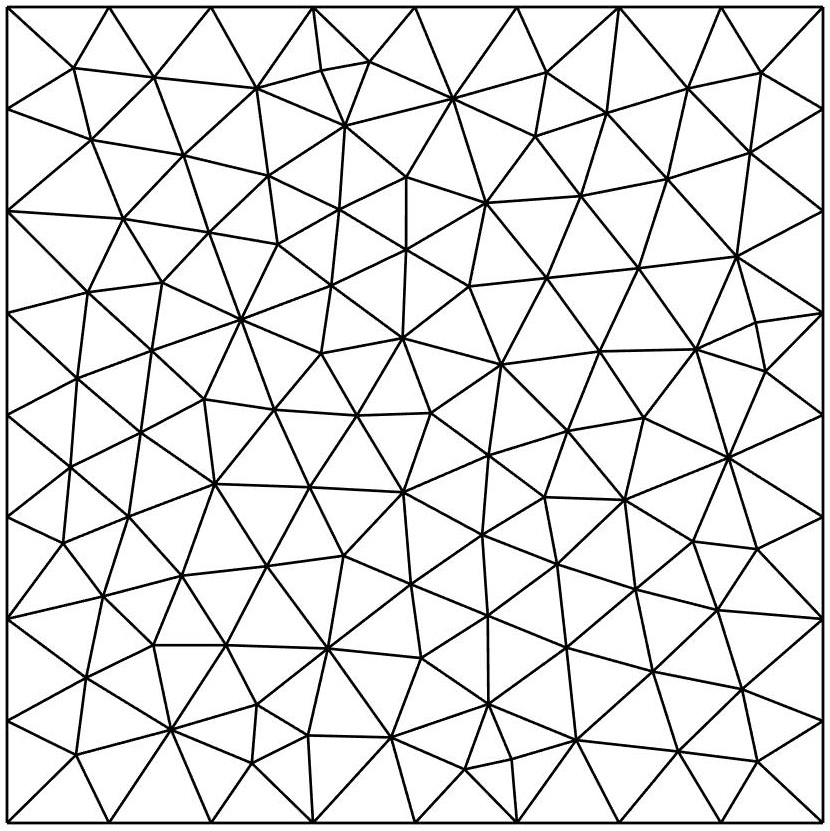}} \\       
      \subfigure[CVT           (\textsf{CVT})  mesh]{\includegraphics[scale=0.125]{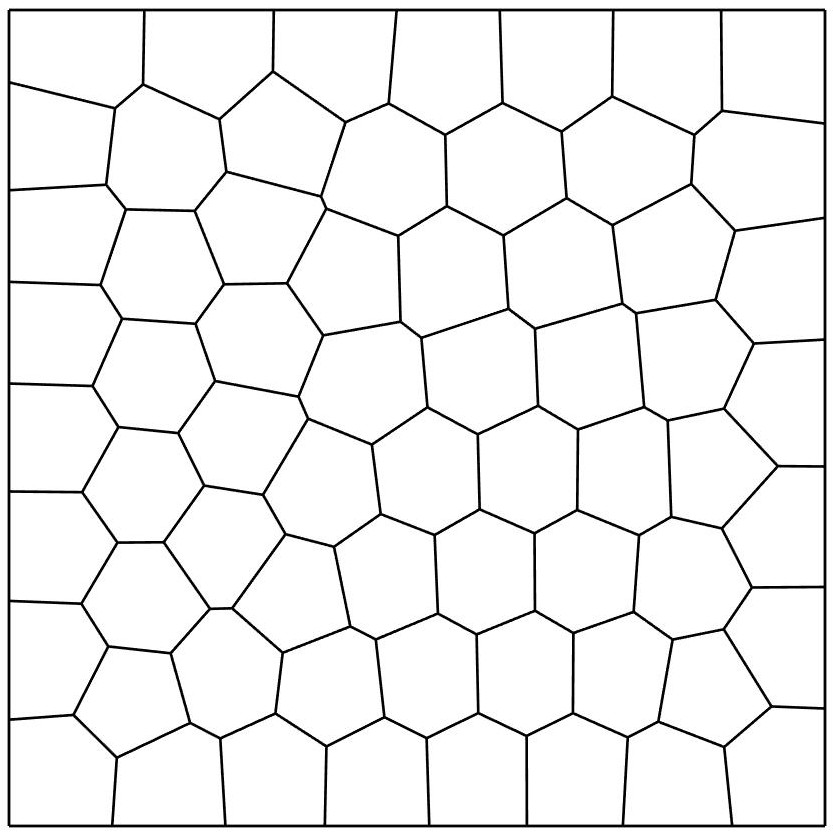}} &\qquad  
      \subfigure[Exagonal      (\textsf{HEX})  mesh]{\includegraphics[scale=0.125]{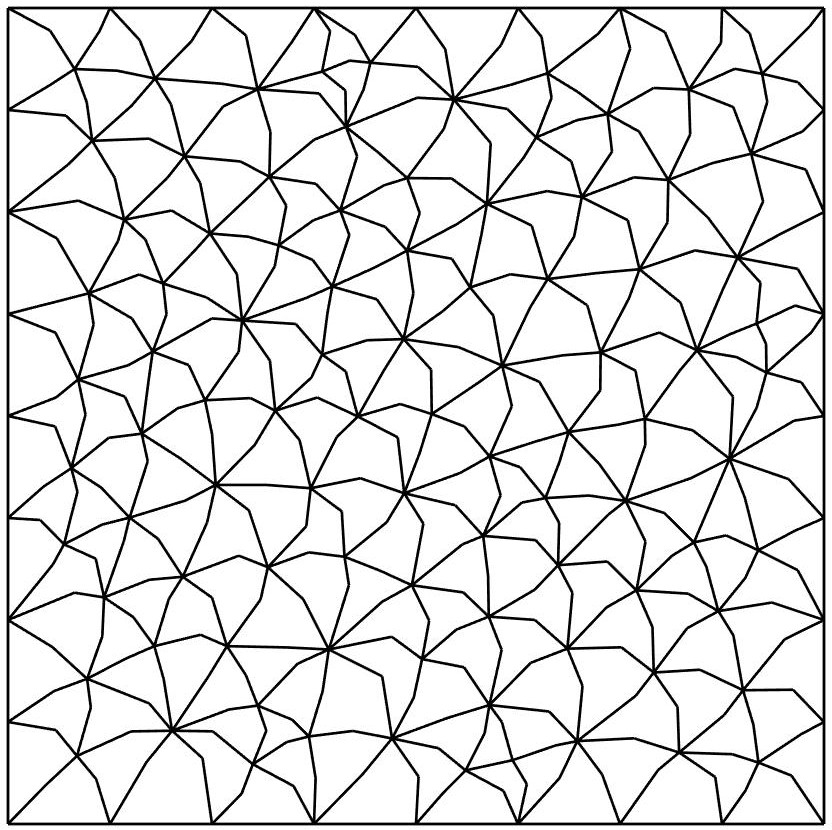}} \\       
    \end{tabular}
  \end{center}
  \caption{Examples of polygonal meshes used in the numerical tests of
    Section~\ref{sec5:numerics}: a quadrilateral (\textsf{QUAD}),
    triangular (\textsf{TRI}), central Voronoi \textsf{CVT}, and
    hexagonal (\textsf{HEX}) mesh.}\label{fig_mesh}
\end{figure}
\renewcommand{\TABROW}[6]{ #1 & #2 & #3 & #4 & #5 & #6 }
\begin{table}
  \begin{center}
    \def\arraystretch{1}\tabcolsep=5pt
    \begin{tabular}{l|rr|rr|rr|rr}
      \hline
      \TABROW{$1\slash{\hh}$      }{    8 }{  16 }{   32 }{    64 }{   128 }\\
      \hline
      \TABROW{\textsf{QUAD}}{   64 }{ 256 }{ 1024 }{  4096 }{ 16384 }\\
      \TABROW{\textsf{TRI} }{  212 }{ 870 }{ 3486 }{ 14080 }{ 56932 }\\
      \TABROW{\textsf{CVT} }{   64 }{ 256 }{ 1024 }{  4096 }{ 16384 }\\
      \TABROW{\textsf{HEX}}{  212 }{ 870 }{ 3486 }{ 14080 }{ 56932 }\\
      \hline
    \end{tabular}
  \end{center}
  \caption{Number of elements of the sequences of meshes versus the
    inverse of the mesh sixe $h$.}
  \label{tab:meshes:elements}
\end{table}

To illustrate the two stabilization strategies that we are going to
test in practice, we rewrite equation~\eqref{eq:poly:ah:def} in matrix
form, i.e.,
\begin{align}
  \matA_{\P} = \matM_{\P} +\matS_{\P},
  \label{matrx_form}
\end{align}
where
$\matA_{\P}$ is the elemental stiffness matrix,
$\matM_{\P}$ is the consistency matrix associated with
$\asP(\PiPr{r}\ush,\PiPr{r}\vsh)$ and
$\matS_{\P}$ is the stabilization matrix associated with
$\SP(\ush-\PiPr{r}\ush,\vsh-\PiPr{r}\vsh)$.
The matrix $\matS_{\P}$ has the following structure:
\begin{align}
  \matS_{\P} =
  \alpha_{\mathsf{stab}}\big( \matI-\matD\matQ )^T \matU ( \matI-\matQ ),
  \label{stab_matrx_form}
\end{align}
where
$\alpha_{\mathsf{stab}}$ is a scalar factor ensuring that matrices
$\matM_{\P}$ and $\matS_{\P}$ have the same scaling with respect to
$\hh$;
$\matI$ is the identity matrix;
$\matQ$ is the matrix representation of the polynomial projection
operator $\PiPr{r}$ with respect to the set of the canonical basis
functions of the virtual element space; $\matD$ is matrix collecting
the degrees of freedom of the polynomial basis chosen in the virtual
element space on its column;
and, finally, $\matU$ is a suitable matrix that allows us to change
the Virtual Element stabilization.
In particular, we consider the following two possible choices of
$\matU$ given in
\begin{itemize}
\item $\matU=\matI$, which is sometimes called in the virtual element
  jargon the ``\emph{dofi-dofi stabilization}'';
\item $\matU=\matD^{\perp}=\matI-\matD(\matD^T\matD)^{-1}\matD^T$.
\end{itemize}
In the second choice above, we use the symbol $\matD^{\perp}$ to
outline the fact that this matrix operator is the orthogonal projector
onto the complement of the vector space spanned by the columns of
$\matD$ (so, we can call it the ``\emph{D-perp stabilization}'').
Since $\matD$ is a maximum rank matrix by definition, the square
matrix $\matD^T\matD$ is a square non-singular matrix, and thus matrix
the $\matU$ is well-defined.
Other possible stabilization strategies for the VEM can be designed
according to Reference~\cite{Dassi-Mascotto:2018,Mascotto:2018}.

\medskip
In the solution of the Poisson equation (i.e.$p_1=1$, cf. Section
\ref{sec:numerics:Poisson}) we take
$\alpha_{\mathsf{stab}}=\text{Trace}(\matM_{\P})\slash{\Ndofs}$, where
$\Ndofs$ is the number of rows/columns of matrix $\matM_{\P}$,
i.e. the local number of degrees of freedom.
In the solution of the biharmonic equation (i.e. $p_1=2$, cf. Section
\ref{sec:numerics:Biharmonic}), the factor $\alpha$ must scale as
$\hh^{-2}$, and the choice is not unique.
To our purpose, we consider three different choices of this parameter,
which we will detail in subsection~\ref{sec:numerics:Biharmonic}.

In each test case, we compare the condition number of the stiffness
matrix and the accuracy of the resulting approximation by measuring
the error in the energy norm and in the $\LTWO$-norm (Poisson
equation) and in the energy norm and in the $\LINF$-norm (biharmonic
equation).
We point out that the experimental estimation of the
condition number of the global stiffness matrix $\matA$ has been
obtained by exploiting the analogies between the Lanczos technique and
the Conjugate Gradient method.
Indeed, within the Conjugate Gradient algorithm we can build a
suitable tridiagonal matrix whose extreme eigenvalues converge to the
extreme eigenvalues of $\matA$, see
Reference~\cite{Golub-VanLoan:1996}, Sections.~9.3 and 10.2 for more
details.
The approximation error is evaluating by computing
$\es_h=\us-\Pi^{p_1}_{r}\ush$, and its energy norm is provided by
$\big(a_{p_1,h}(\es_h,\es_h))^{1\slash{2}}$.
Finally, it is informative to say that we carried out all the tests of
this section by using our in-house $\Cs^{++}$ and
MATLAB~\cite{MATLAB:2020} implementations.

\subsection{Poisson equation}
\label{sec:numerics:Poisson}

We recall that, for fixed $p_2 \geq p_1$, $k=p_2-1$ denotes the
$\CS{k}$-regularity of the global virtual element space, and that $r$
denotes the degree of the polynomials contained in each elemental
approximation space.
We carried out the calculations corresponding to the two following
test cases (TCs):
\begin{itemize}
\item \emph{TC1}: $(k=0,r=2)$, $(k=1,r=2)$;
\item \emph{TC2}: $(k=0,r=3)$, $(k=1,r=3)$, $(k=2,r=3)$.
\end{itemize}
For both test cases, we consider the four different mesh families
shown in Fig.\ref{fig_mesh} and the two possible choices of the
stabilizing bilinear form discussed above, i.e., by choosing
$\matU=\matI$ and $\matU=\matD^{\perp}$ in~\eqref{stab_matrx_form}.
In every calculation, we measure the error in the energy norm
($\HS{1}$-norm) and in the $\LTWO$-norm , and we evaluate the
condition number of the linear system of equations.
The plots of the error curves versus $\hh$ (loglog scale) are shown in
Figs.~\ref{fig:poisson:C012-P3:hmax} and ~\ref{fig:poisson:C01-P2:hmax}.
The computed condition numbers are reported in
Tables~\ref{tab:poisson:condnum:C01-P2}
and~\ref{tab:poisson:condnum:C012-P3}.
We observe that the two different stabilizations seems to provide
comparable results concerning the condition numbers, which is
exhibiting the expected growth $O(h^{-2})$, although the choice
$\matU=\matD^{\perp}$ seems to provide lower condition numbers for the
VEM with higher regularity.

The behavior of the error curves is also very similar for all these
variants of the VEM, the error curves being very closed in almost
every plot and overlapping to the point that they cannot easily be
distinguished.
Optimal convergence rates are seen in every plot.
We recall that the error in energy norm is expected to decrease
proportionally to $\hh^{m}$ for $\hh\to0$ for all values of teh
polynomial order (order of accuracy of the method) $m$ here
considered.
Instead, the error in the $\LS{2}$-norm is expected to reduce as
$\hh^{2}$ for $r=2$ and $\hh^{4}$ for $r=4$, as we do not adopted the
modified, e.g., ``enhanced'', version of the
VEM~\cite{Ahmad-Alsaedi-Brezzi-Marini-Russo:2013}, which makes a better
approximation to the solution possible for the low order case $r=2$.
This loss of an order of convergence is a well-known phenomenon and
has been discussed in a previous
article~\cite{BeiraodaVeiga-Manzini:2014}.



\newcommand{\CzeroPtwo}  { \multicolumn{2}{c|}{$\CS{0}-\PS{2}$} }
\newcommand{\ConePtwo}   { \multicolumn{2}{c} {$\CS{1}-\PS{2}$} }

\newcommand{\CzeroPthree}{ \multicolumn{2}{c|}{$\CS{0}-\PS{3}$} }
\newcommand{\ConePthree} { \multicolumn{2}{c|}{$\CS{1}-\PS{3}$} }
\newcommand{\CtwoPthree} { \multicolumn{2}{c} {$\CS{2}-\PS{3}$} }

  
\newcommand{\MGT}[1]{{\color{magenta}#1}} 

\newcommand{\qm}{\RED{\textbf{??}}}
\newcommand{\OK}{\BLUE{\textbf{OK!}}}
\newcommand{\CHECK}{\MGT{Check!}}


\renewcommand{\TABROW}[5]{ #1 & #2 & #3 & #4 & #5 }
\begin{table}
  \begin{center}
    \def\arraystretch{1}\tabcolsep=5pt
    \begin{tabular}{r|cc|cc}
      \hline
      {}              &          \CzeroPtwo &           \ConePtwo \\
      \TABROW{ $1/h$ }{  \stabI }{  \stabD }{  \stabI }{  \stabD }\\
      \hline
      {}              & \multicolumn{4}{c}{\text{\textsf{QUAD} meshes}}\\
      \hline
      \TABROW{     8 }{ 1.24e+3 }{ 4.55e+2 }{ 2.38e+3 }{ 3.04e+2 }\\
      \TABROW{    16 }{ 4.99e+3 }{ 1.83e+3 }{ 1.04e+4 }{ 1.29e+3 }\\
      \TABROW{    32 }{ 2.00e+4 }{ 7.35e+3 }{ 4.31e+4 }{ 5.28e+3 }\\
      \TABROW{    64 }{ 8.01e+4 }{ 2.94e+4 }{ 1.75e+5 }{ 2.13e+4 }\\
      \TABROW{   128 }{ 3.21e+5 }{ 1.18e+5 }{ 7.05e+5 }{ 8.59e+4 }\\
      \hline
      {}              & \multicolumn{4}{c}{\text{\textsf{TRI} meshes}}\\
      \hline
      \TABROW{     8 }{ 9.97e+3 }{ 1.37e+3 }{ 1.59e+4 }{ 1.09e+3 }\\
      \TABROW{    16 }{ 4.27e+4 }{ 5.97e+3 }{ 7.49e+4 }{ 5.14e+3 }\\
      \TABROW{    32 }{ 1.69e+5 }{ 2.37e+4 }{ 3.27e+5 }{ 1.99e+4 }\\
      \TABROW{    64 }{ 6.80e+5 }{ 9.45e+4 }{ 1.31e+6 }{ 8.26e+4 }\\
      \TABROW{   128 }{ 2.79e+6 }{ 3.89e+5 }{ 5.50e+6 }{ 3.34e+5 }\\
      \hline
      {}              & \multicolumn{4}{c}{\textsf{CVT} meshes}\\
      \hline
      \TABROW{     8 }{ 1.53e+3 }{ 6.37e+2 }{ 1.95e+3 }{ 3.69e+2 }\\
      \TABROW{    16 }{ 6.32e+3 }{ 2.62e+3 }{ 8.97e+3 }{ 1.51e+3 }\\
      \TABROW{    32 }{ 2.70e+4 }{ 1.04e+4 }{ 3.97e+4 }{ 6.22e+3 }\\
      \TABROW{    64 }{ 1.02e+5 }{ 4.15e+4 }{ 1.65e+5 }{ 2.47e+4 }\\
      \TABROW{   128 }{ 4.14e+5 }{ 1.64e+5 }{ 6.62e+5 }{ 9.82e+4 }\\
      \hline
      {}              & \multicolumn{4}{c}{\text{\textsf{HEX} meshes}}\\
      \hline
      \TABROW{     8 }{ 8.88e+3 }{ 2.23e+3 }{ 1.10e+4 }{ 1.26e+3 }\\
      \TABROW{    16 }{ 4.08e+4 }{ 9.72e+3 }{ 5.44e+4 }{ 3.70e+3 }\\
      \TABROW{    32 }{ 2.22e+5 }{ 3.88e+4 }{ 2.42e+5 }{ 2.21e+4 }\\
      \TABROW{    64 }{ 8.79e+5 }{ 1.64e+5 }{ 1.23e+6 }{ 9.02e+4 }\\
      \TABROW{   128 }{ 5.57e+6 }{ 6.65e+5 }{ 5.01e+6 }{ 3.73e+5 }\\
      \hline
    \end{tabular}
    \caption{Poisson equation. Comparison of the computed
      condition numbers obtained with the different stabilization
      strategies, different regularity $k=0,1$ and polynomial order $r=2$.}
    \label{tab:poisson:condnum:C01-P2}
  \end{center}
\end{table}

\renewcommand{\TABROW}[7]{ #1 & #2 & #3 & #4 & #5 & #6 & #7 }
\begin{table}
  \begin{center}
    \def\arraystretch{1}\tabcolsep=5pt
    \begin{tabular}{r|rr|rr|rr}
      \hline
      {}              &          \CzeroPthree &         \ConePthree &        \CtwoPthree \\
      \TABROW{ $1/h$ }{    \stabI }{  \stabD }{  \stabI }{  \stabD }{  \stabI }{ \stabD }\\
      \hline
      {}              & \multicolumn{6}{c}{\text{\textsf{QUAD} meshes}}\\
      \hline
      \TABROW{     8 }{  3.74e+4 }{ 1.32e+6 }{ 7.67e+6 }{ 6.70e+3 }{ 8.41e+6 }{ 3.25e+3 }\\
      \TABROW{    16 }{  1.49e+5 }{ 5.26e+6 }{ 2.90e+7 }{ 2.50e+4 }{ 5.52e+7 }{ 1.98e+4 }\\
      \TABROW{    32 }{  5.96e+5 }{ 2.10e+7 }{ 1.16e+8 }{ 9.92e+4 }{ 2.50e+8 }{ 8.93e+4 }\\
      \TABROW{    64 }{  2.38e+6 }{ 8.41e+7 }{ 4.66e+8 }{ 3.97e+5 }{ 1.05e+9 }{ 3.72e+5 }\\
      \TABROW{   128 }{  9.53e+6 }{ 3.36e+8 }{ 1.87e+9 }{ 1.59e+6 }{ 4.30e+9 }{ 1.52e+6 }\\
      \hline
      {}              & \multicolumn{6}{c}{\text{\textsf{TRI} meshes}}\\
      \hline
      \TABROW{     8 }{ 4.88e+5 }{ 1.26e+8  }{ 2.74e+9     }{ 1.72e+5 }{ 5.91e+9     }{ 1.05e+5 }\\
      \TABROW{    16 }{ 2.43e+6 }{ 6.87e+8  }{ 1.56e+10    }{ 8.01e+5 }{ 5.39e+10    }{ 6.53e+5 }\\
      \TABROW{    32 }{ 1.15e+7 }{ 3.73e+9  }{ 9.67e+10    }{ 3.75e+6 }{ 3.60e+11    }{ 3.74e+6 }\\
      \TABROW{    64 }{ 4.14e+7 }{ 1.25e+10 }{ 2.90e+11    }{ 1.35e+7 }{ \text{n.a.} }{ 1.27e+7 }\\ 
      \TABROW{   128 }{ 2.15e+8 }{ 7.39e+10 }{ \text{n.a.} }{ 6.82e+7 }{ \text{n.a.} }{ 6.82e+7 }\\
      \hline
      {}              & \multicolumn{6}{c}{\textsf{CVT} meshes}\\
      \hline
      \TABROW{     8 }{ 8.79e+4 }{ 4.98e+6  }{ 3.31e+7  }{ 2.14e+4 }{ 3.83e+7  }{ 1.34e+4 }\\
      \TABROW{    16 }{ 3.91e+5 }{ 2.22e+7  }{ 1.95e+8  }{ 1.05e+5 }{ 2.91e+8  }{ 8.55e+4 }\\
      \TABROW{    32 }{ 1.56e+6 }{ 8.23e+7  }{ 8.29e+8  }{ 3.65e+5 }{ 1.50e+9  }{ 3.27e+5 }\\
      \TABROW{    64 }{ 8.03e+6 }{ 3.45e+8  }{ 7.62e+9  }{ 2.20e+6 }{ 1.51e+10 }{ 2.08e+6 }\\
      \TABROW{   128 }{ 2.69e+7 }{ 3.11e+8  }{ 2.16e+10 }{ 7.85e+6 }{ 5.00e+10 }{ 7.55e+6 }\\
      \hline
      {}              & \multicolumn{6}{c}{\text{\textsf{HEX} meshes}}\\
      \hline
      \TABROW{     8 }{ 8.94e+5 }{ 1.87e+8  }{ 3.07e+8  }{ 1.24e+5 }{ 9.13e+8     }{ 9.65e+4 }\\
      \TABROW{    16 }{ 4.82e+6 }{ 1.50e+9  }{ 1.80e+9  }{ 6.50e+5 }{ 6.37e+9     }{ 5.75e+5 }\\
      \TABROW{    32 }{ 2.56e+7 }{ 8.47e+9  }{ 1.87e+10 }{ 3.55e+6 }{ 1.65e+10    }{ 3.33e+6 }\\
      \TABROW{    64 }{ 8.98e+7 }{ 1.99e+10 }{ 8.91e+10 }{ 1.57e+7 }{ 2.58e+11    }{ 1.51e+7 }\\
      \TABROW{   128 }{ 1.16e+8 }{ 1.05e+11 }{ 3.24e+11 }{ 5.67e+7 }{ \text{n.a.} }{ 5.53e+7 }\\
      \hline
    \end{tabular}
    \caption{Poisson equation. Comparison of the condition numbers
      obtained with the different stabilization strategies, different
      regularity $k=0,1,2$ and polynomial order $r=3$. The acronym
      ``n.a.'' stands for ``not available'' since the resulting linear
      system was too badly-conditioned to estimate the condition number.}
    \label{tab:poisson:condnum:C012-P3}
  \end{center}
\end{table}
\begin{figure}[!htb]
  \begin{center}
    \begin{tabular}{cc}
      \includegraphics[width=6cm]{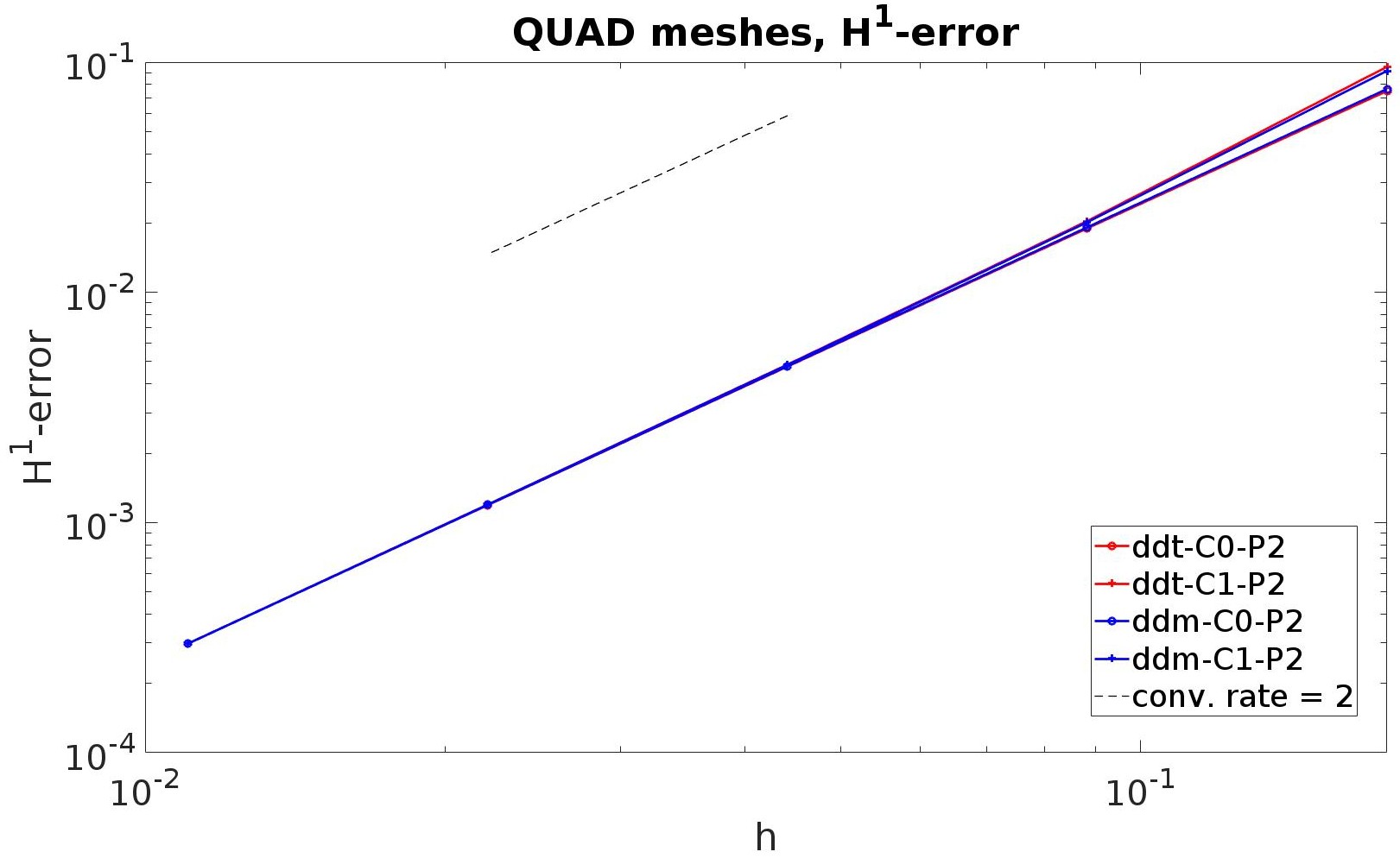}       
      \includegraphics[width=6cm]{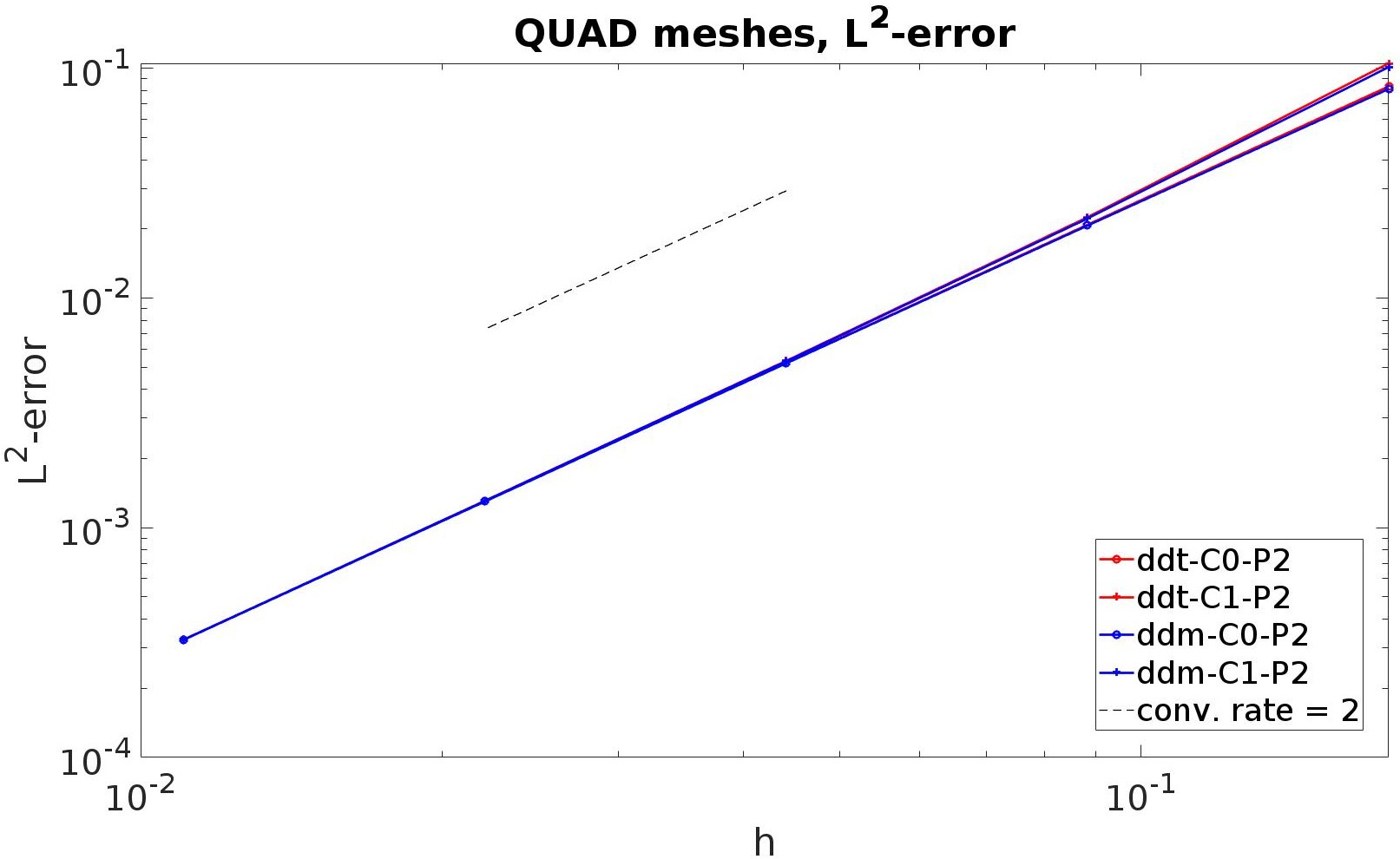}\\     
      \includegraphics[width=6cm]{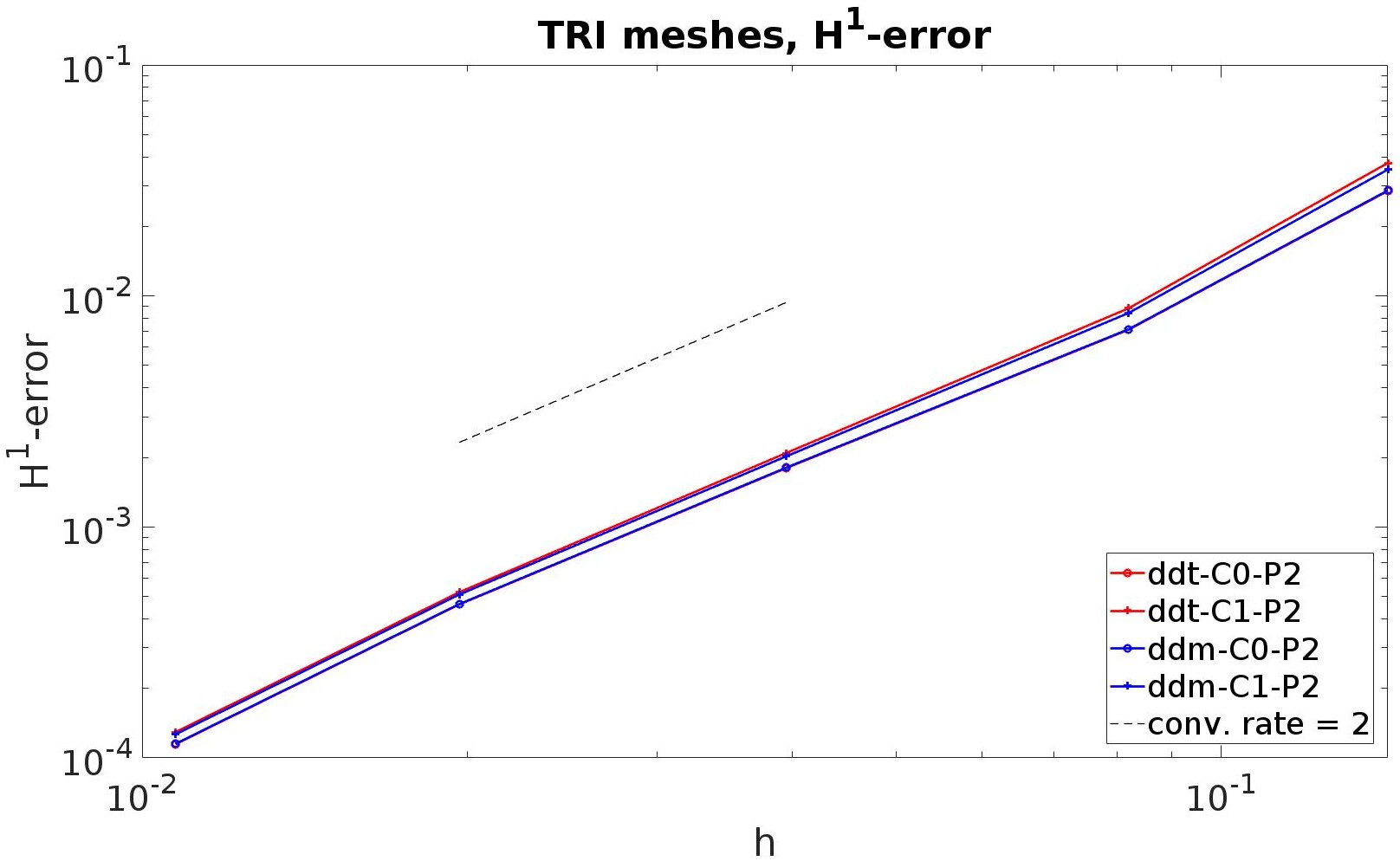}       
      \includegraphics[width=6cm]{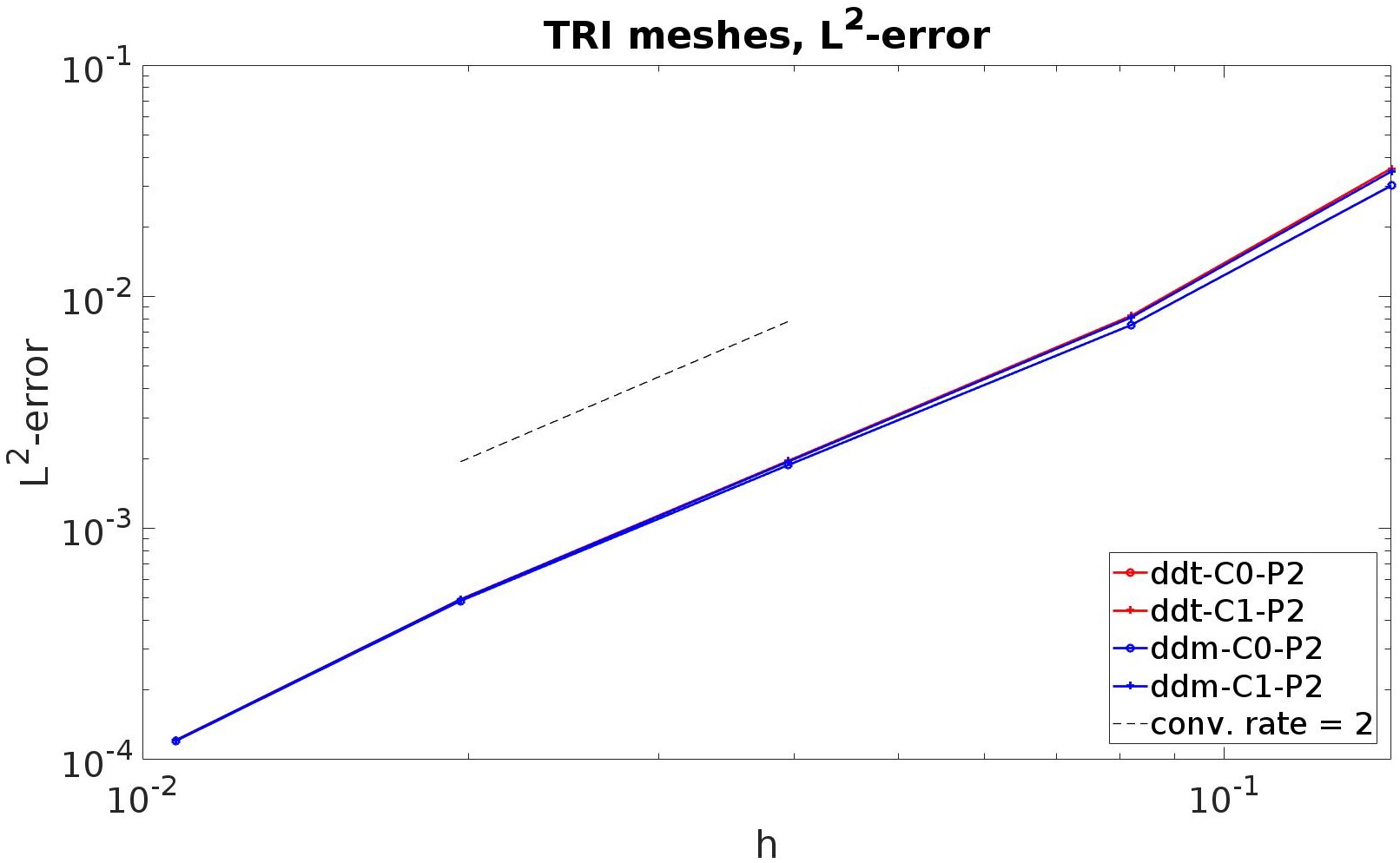}\\     
      \includegraphics[width=6cm]{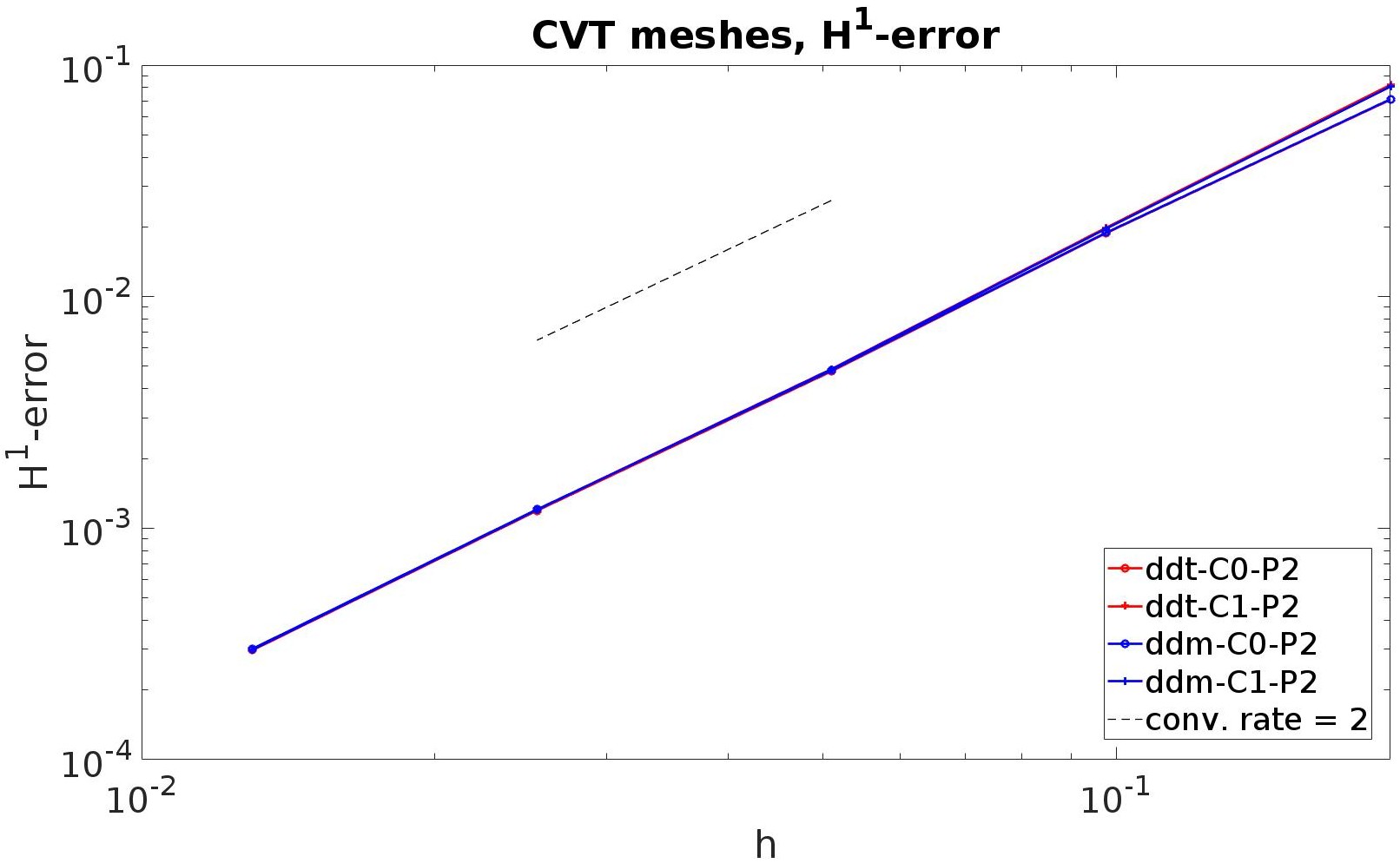}       
      \includegraphics[width=6cm]{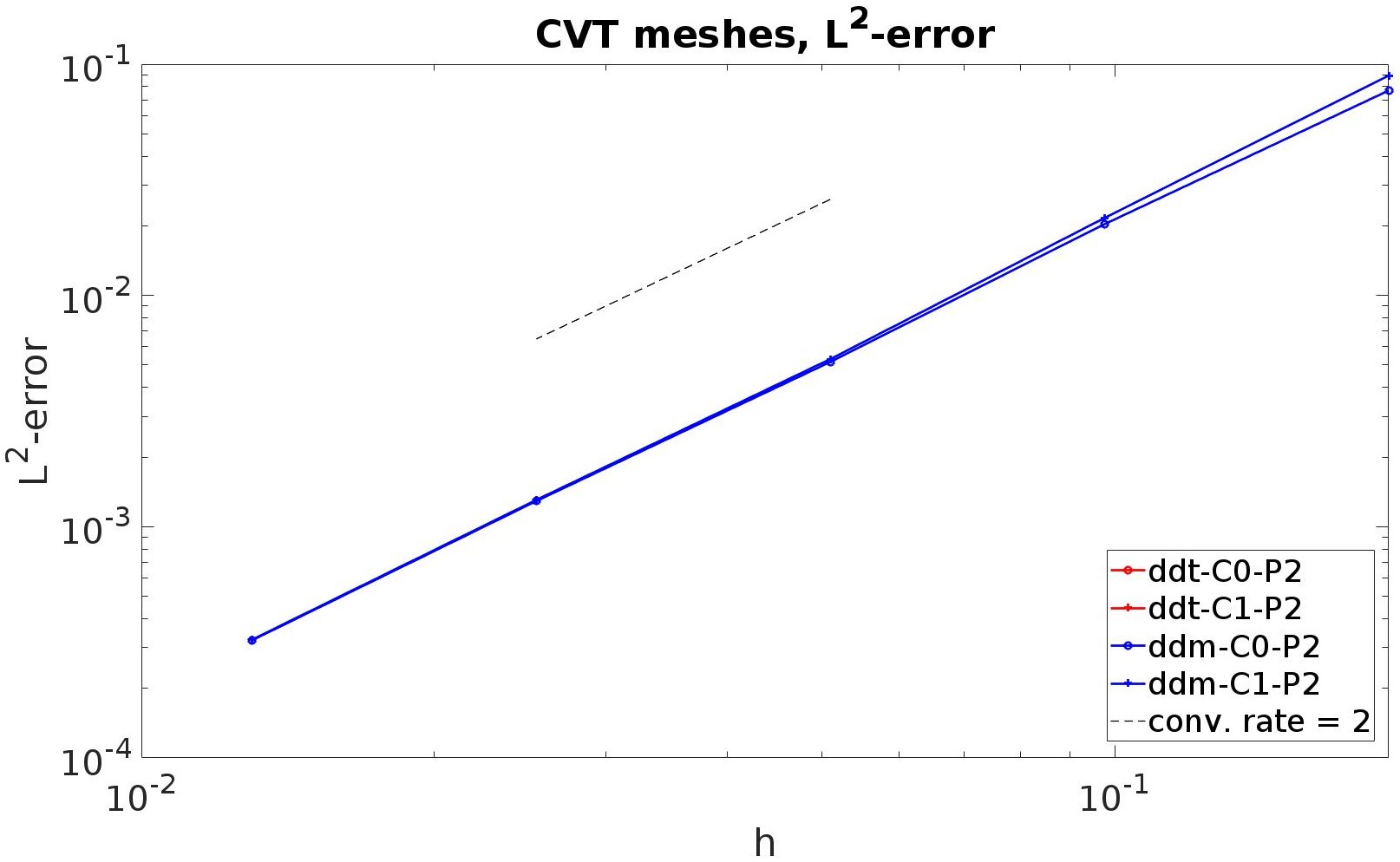}\\     
      \includegraphics[width=6cm]{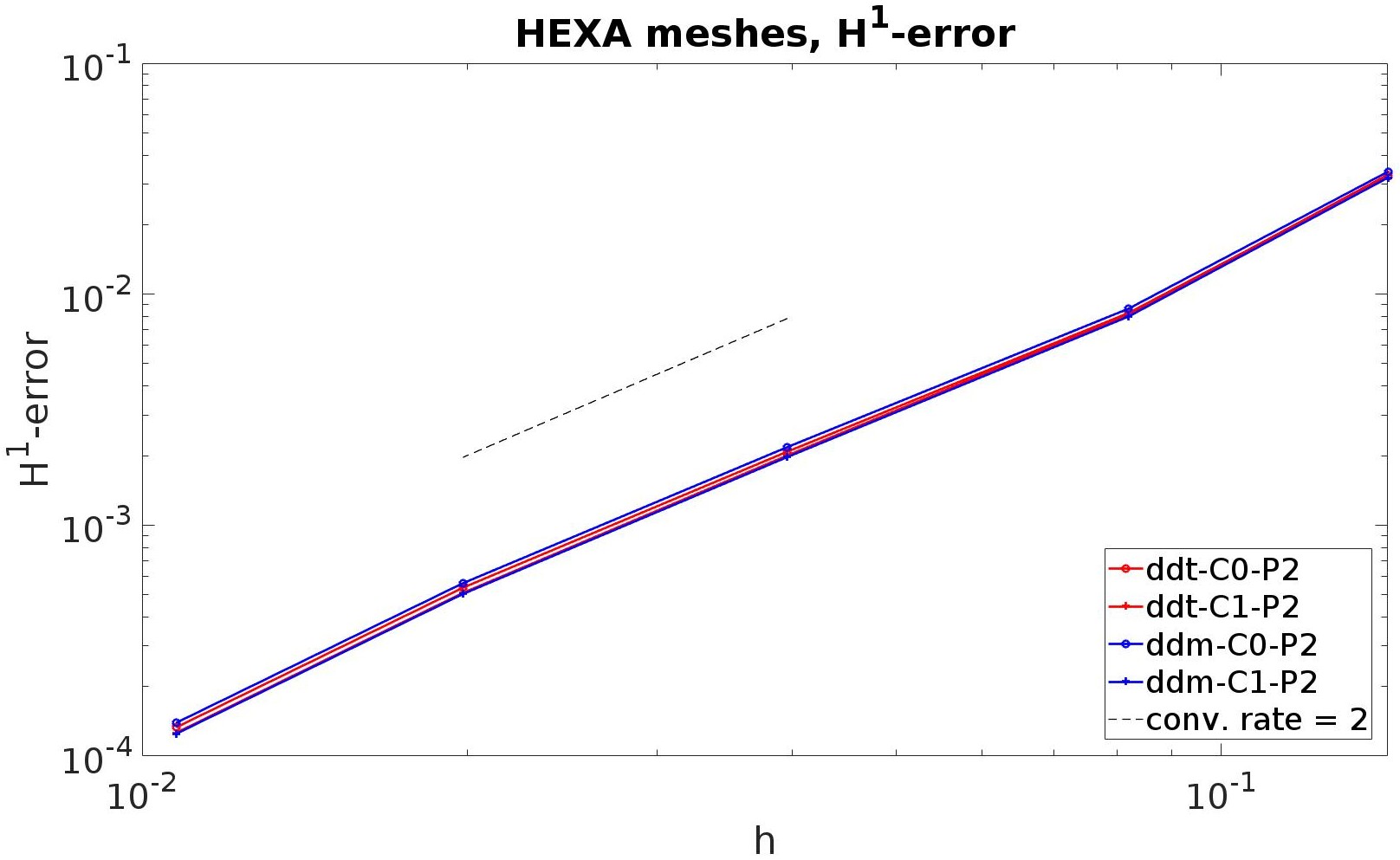}       
      \includegraphics[width=6cm]{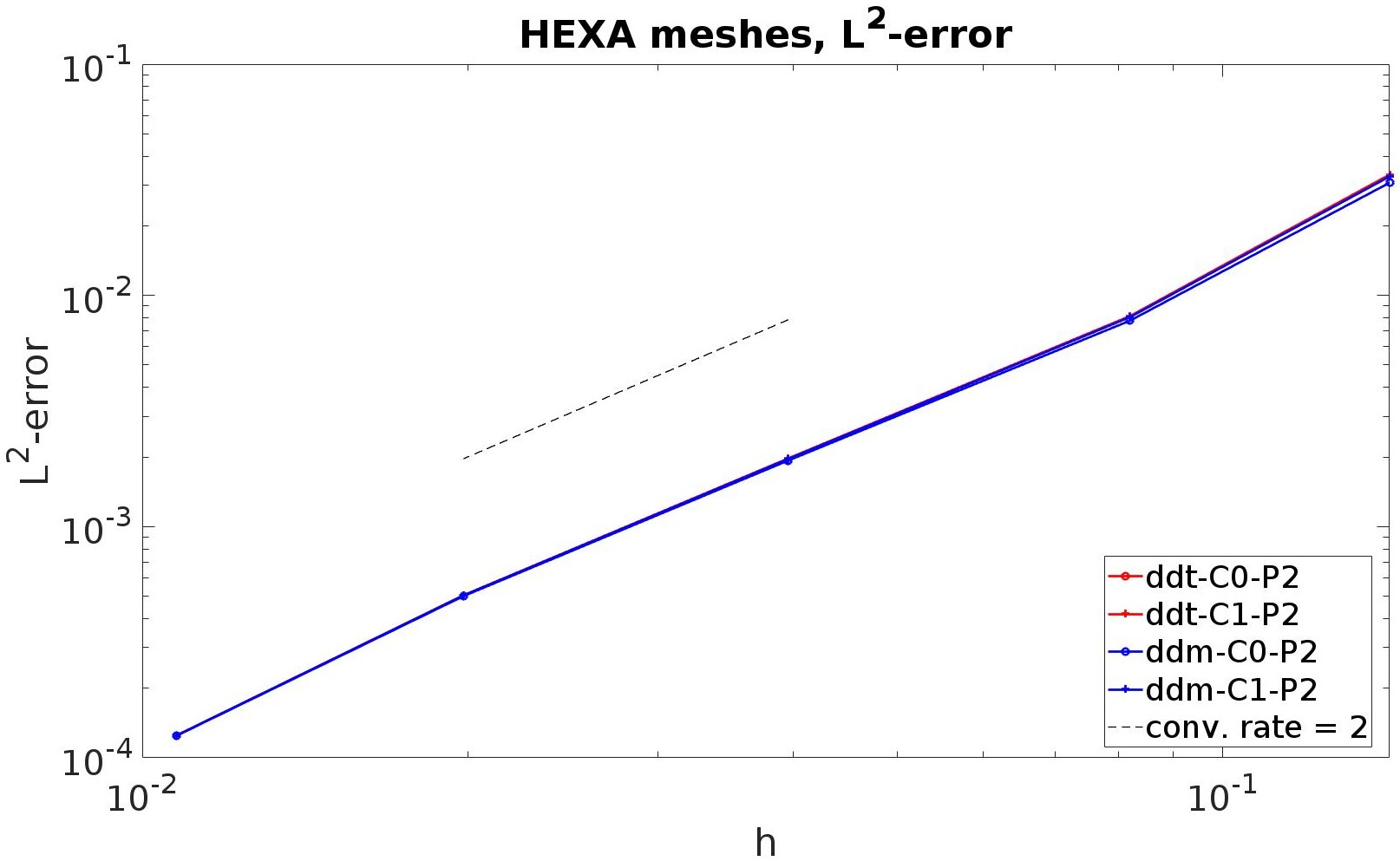}\\     
    \end{tabular}
    \caption{Poisson equation, test case TC1.
      Plots of the error curves versus the mesh size parameter $\hh$
      for the discretization using the (reduced) virtual element space
      of Section~\ref{S:lower} with $p_1=1$, $p_2=1,2$, $r=2$ on
      different polygonal mesh families and stabilization terms.
      The errors are measured using the energy norm (left) and the
      $\LTWO$-norm (right), and are expected to scale proportionally
      to $\hh^2$.  }
    \label{fig:poisson:C012-P3:hmax}
  \end{center}
\end{figure}
\begin{figure}[!htb]
  \begin{center}
    \begin{tabular}{cc}
      \includegraphics[width=6cm]{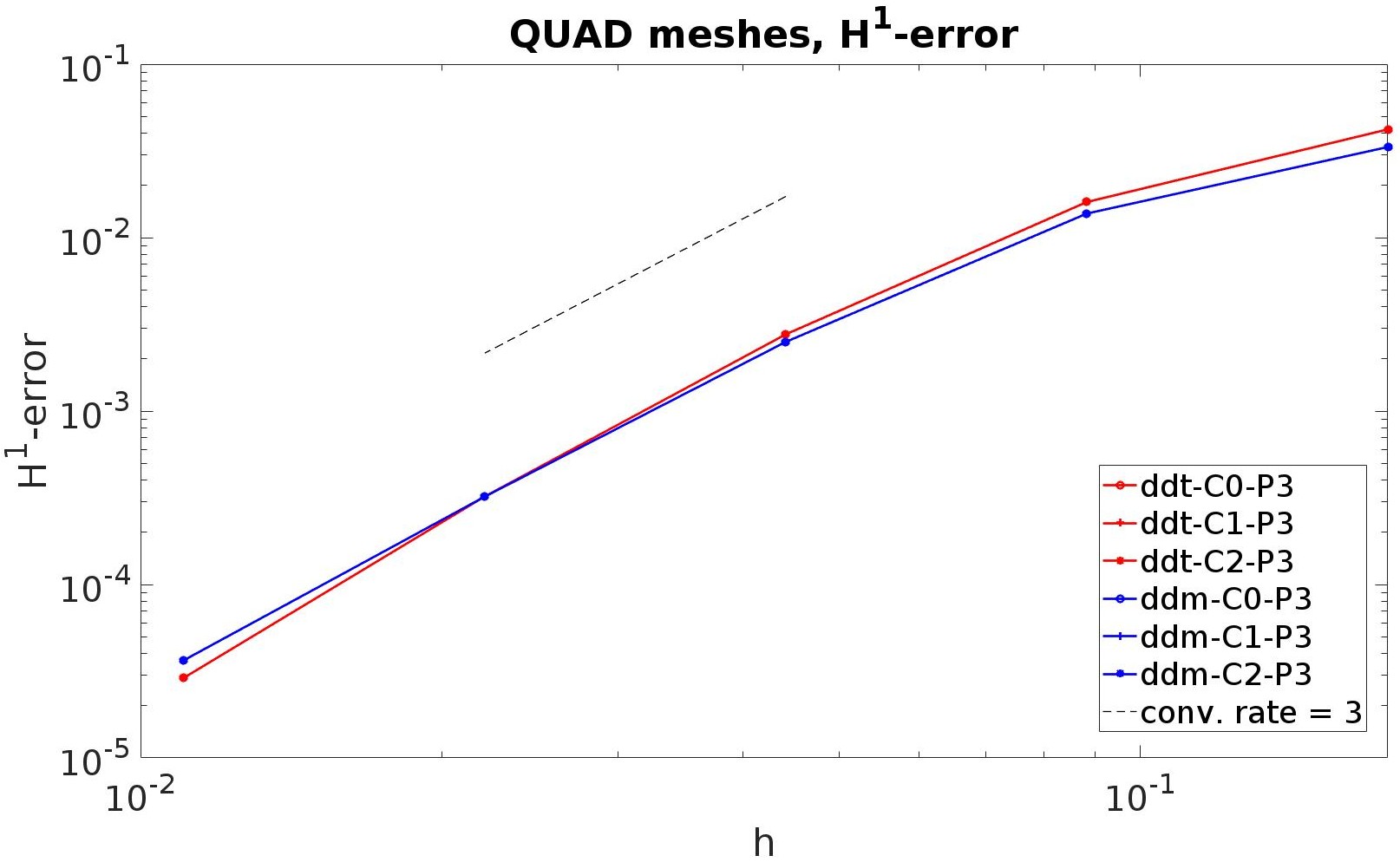}        
      \includegraphics[width=6cm]{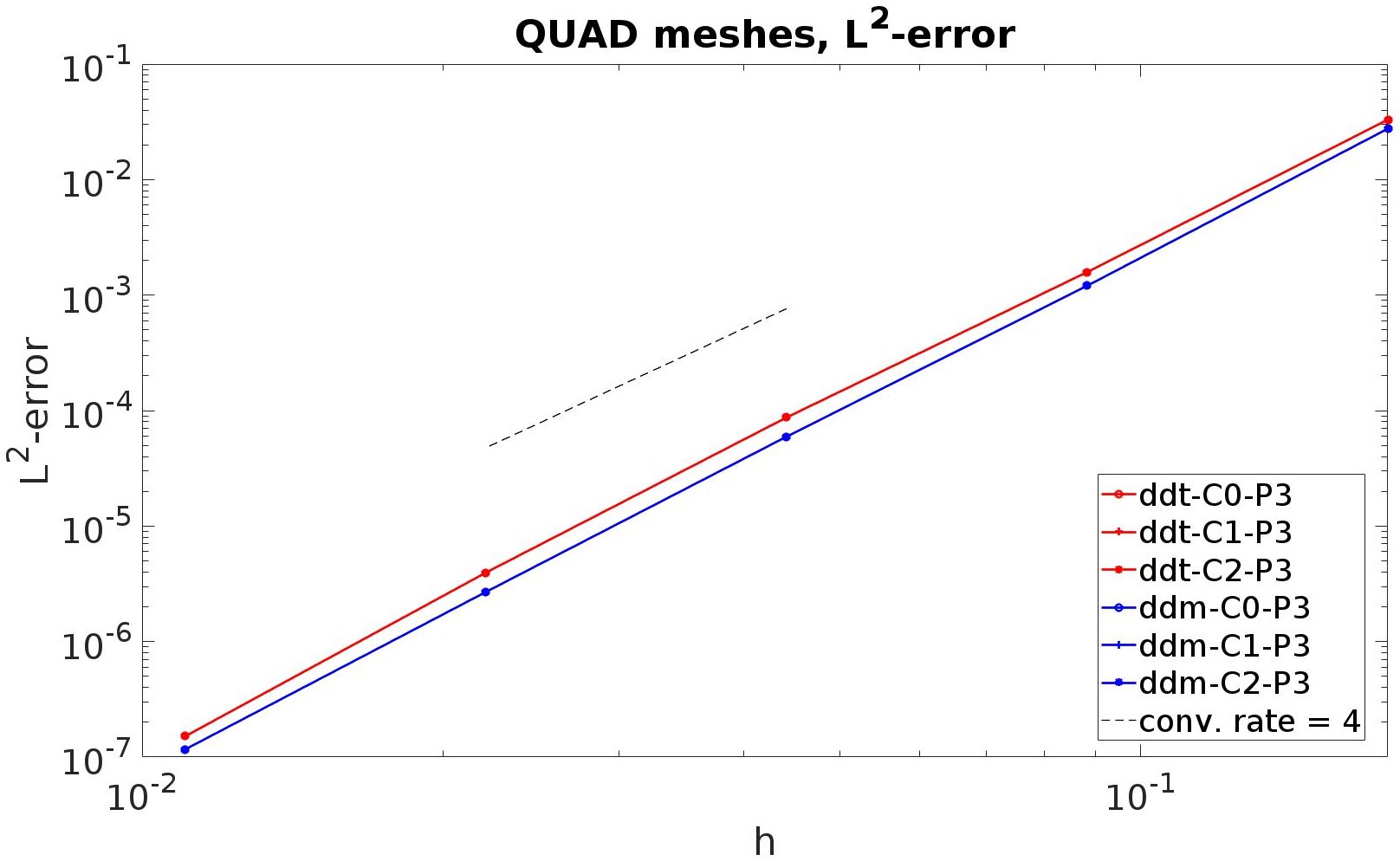}\\      
      \includegraphics[width=6cm]{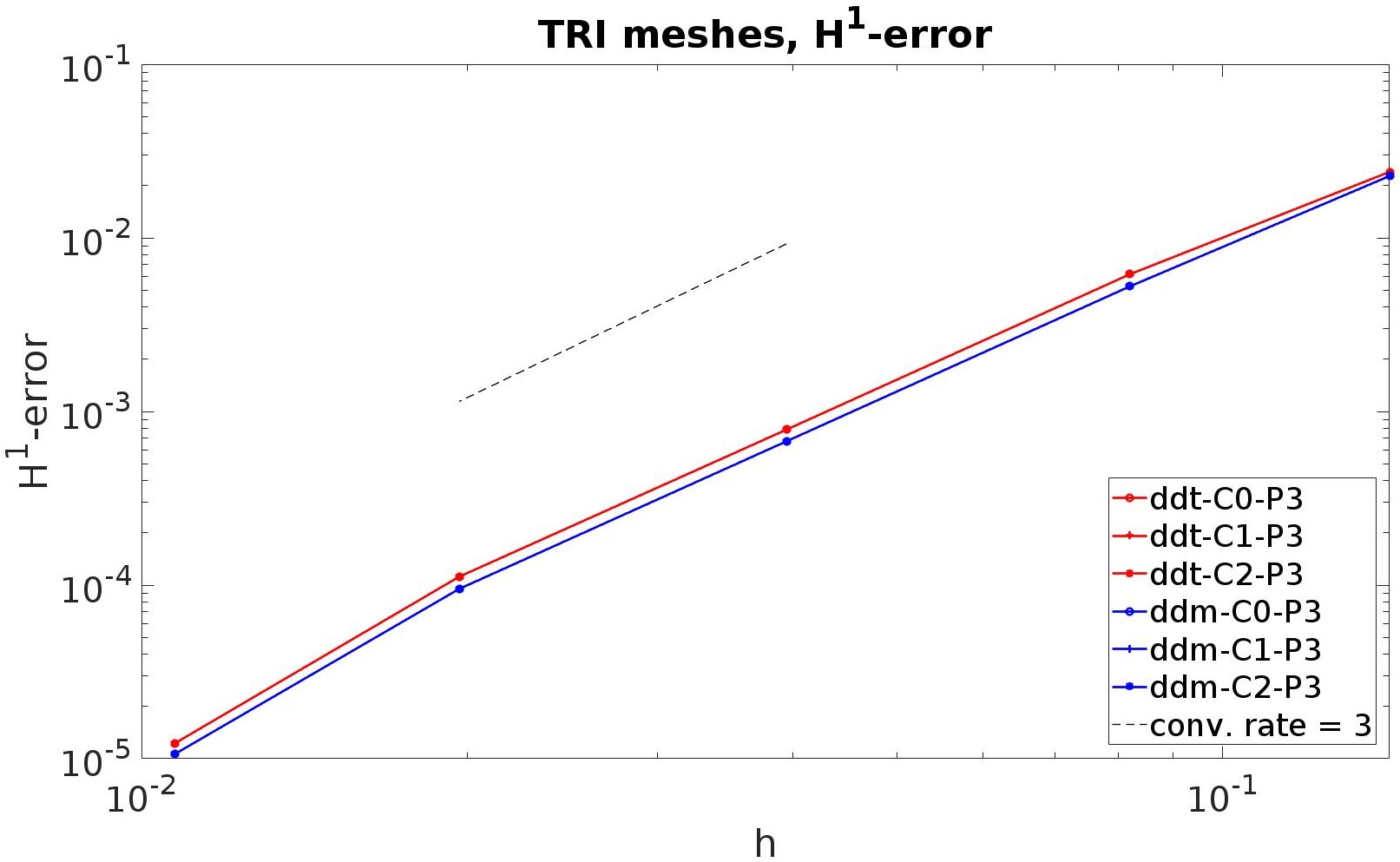}        
      \includegraphics[width=6cm]{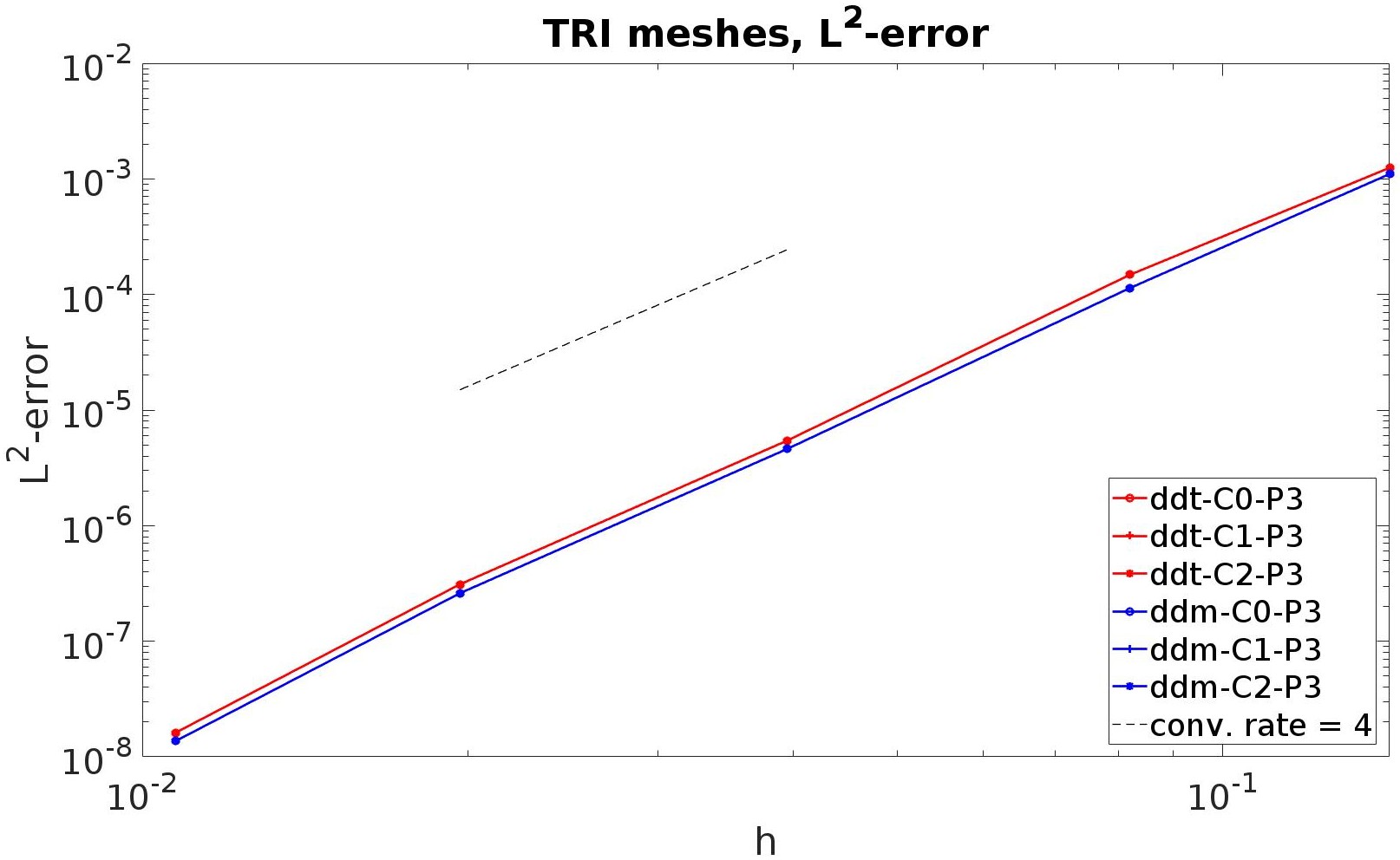}\\      
      \includegraphics[width=6cm]{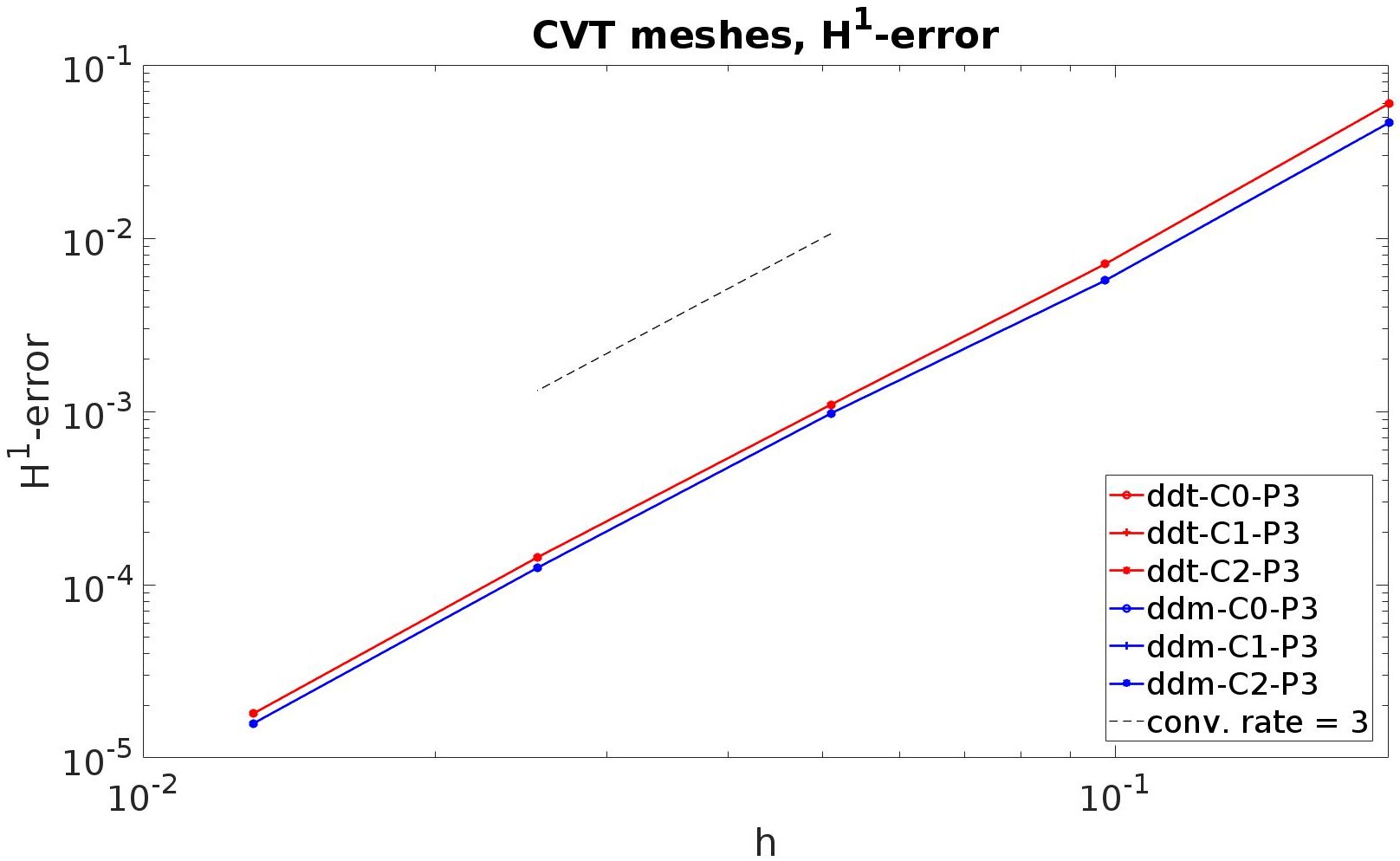}        
      \includegraphics[width=6cm]{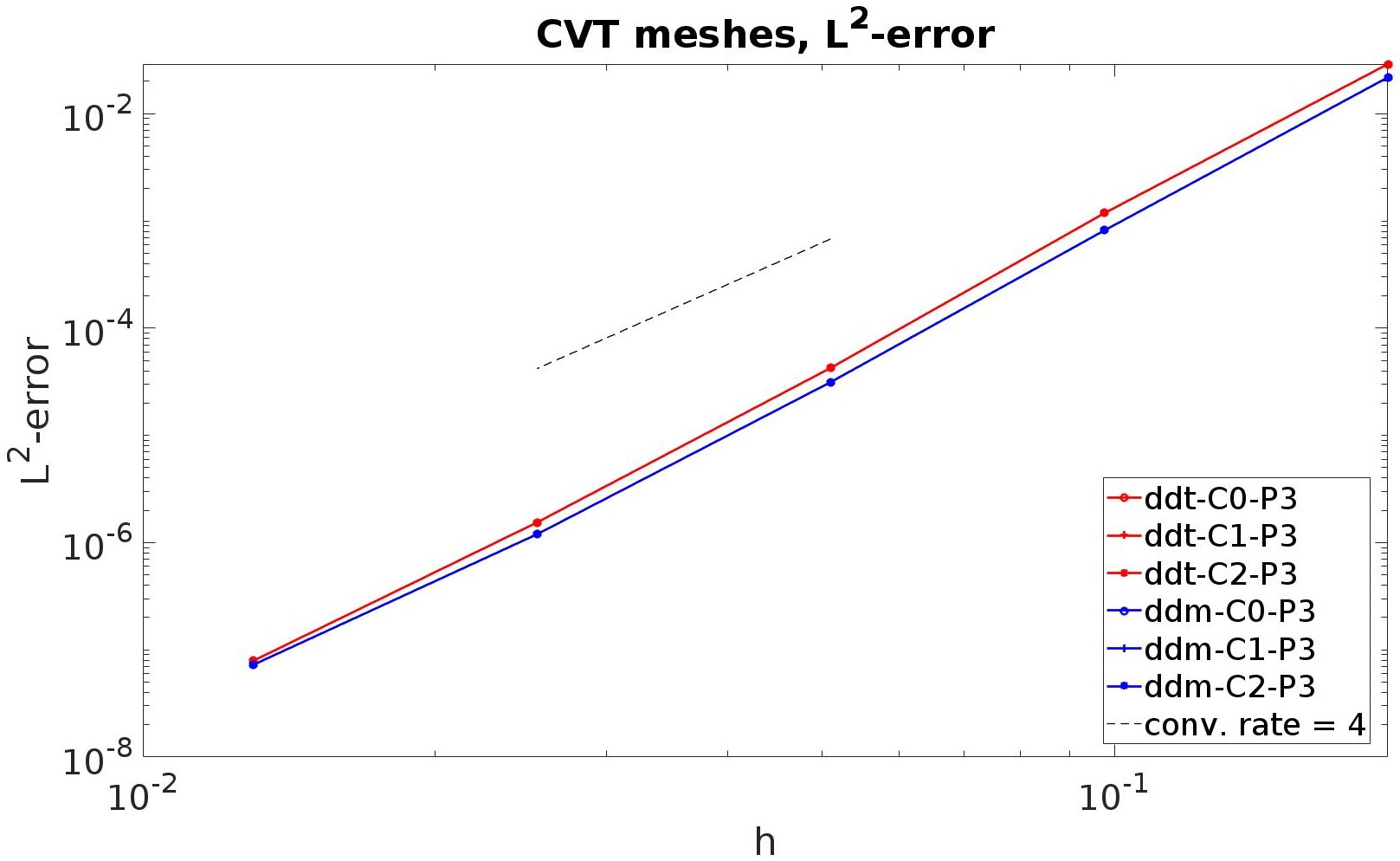}\\      
      \includegraphics[width=6cm]{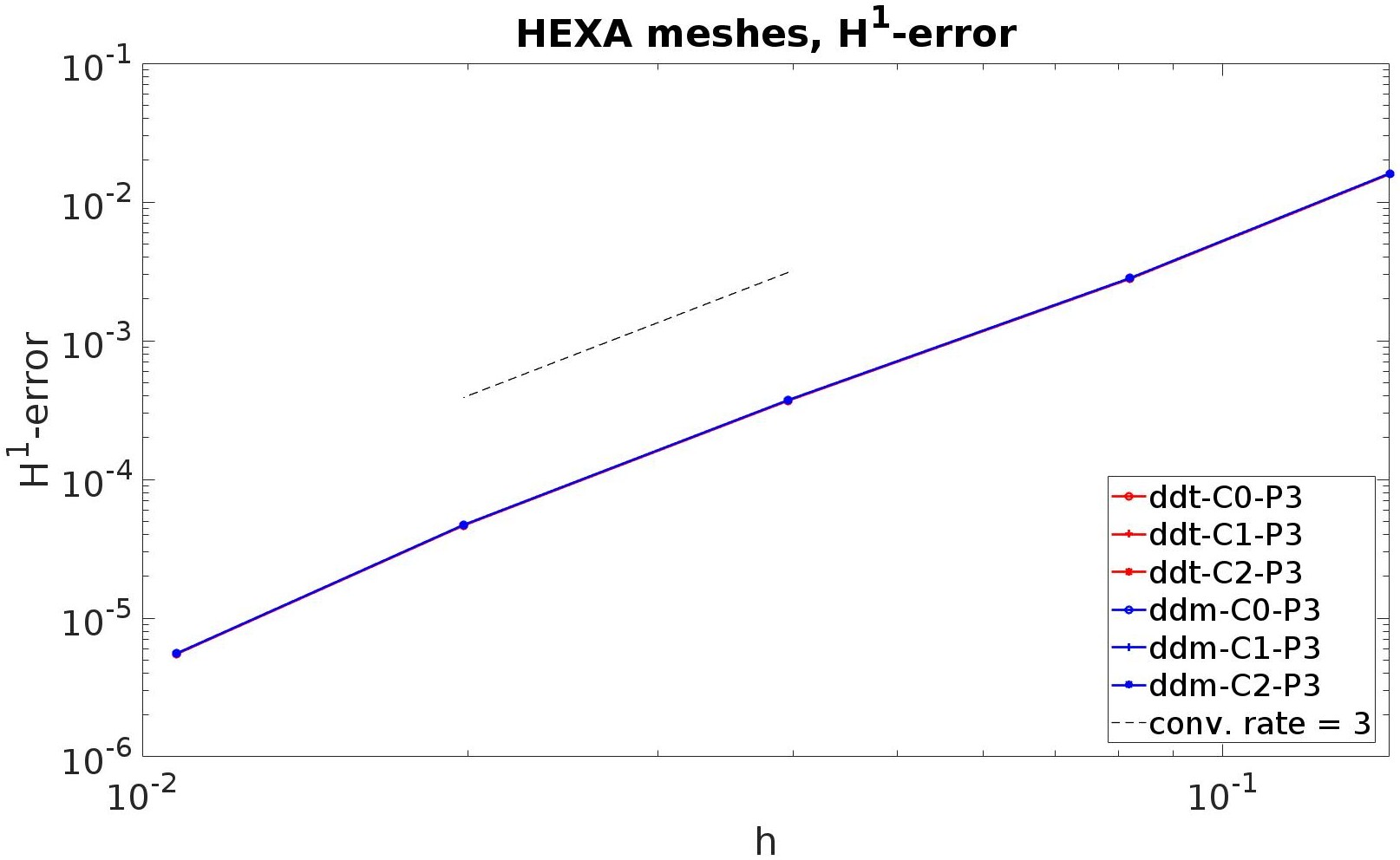}        
      \includegraphics[width=6cm]{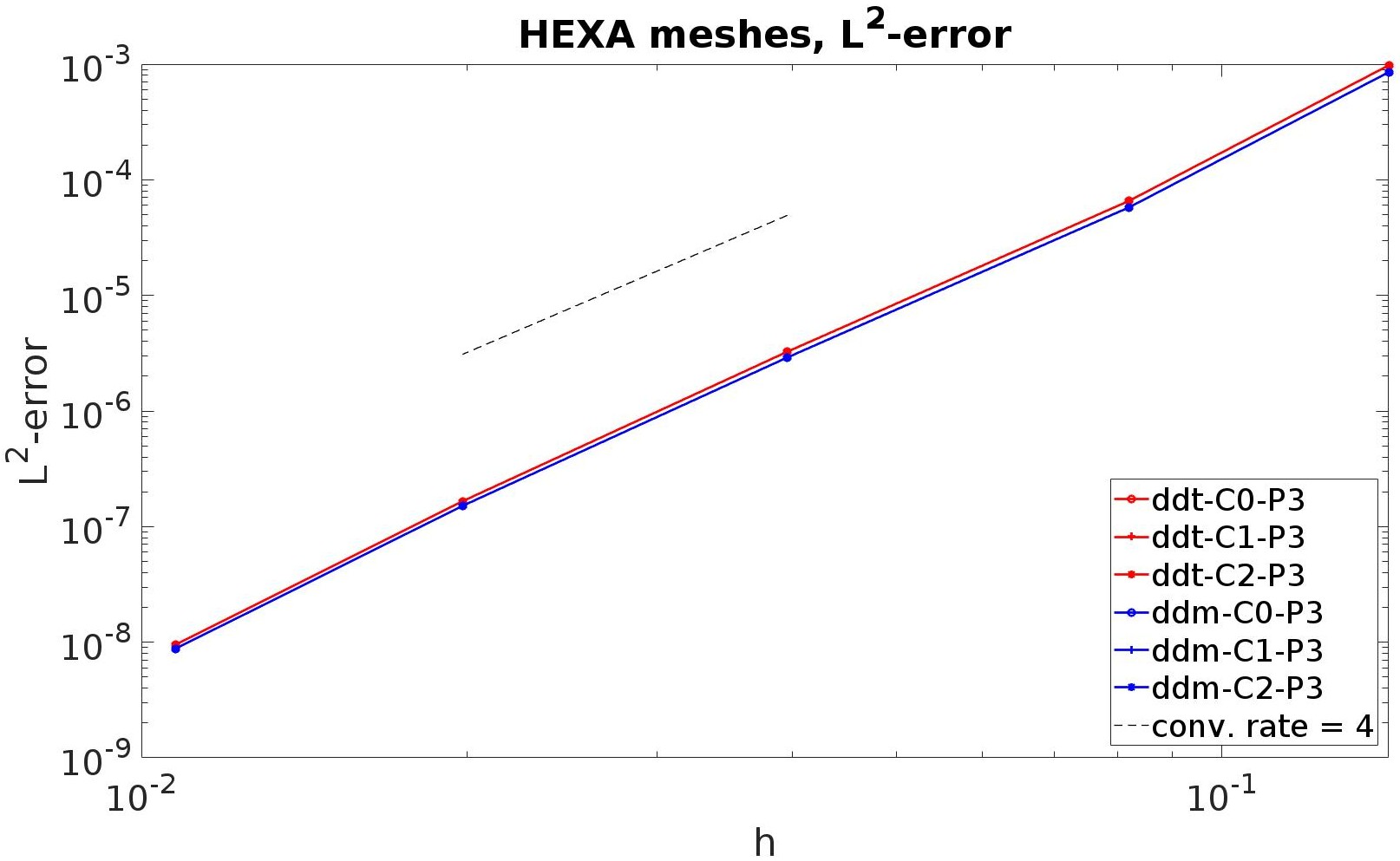}\\      
    \end{tabular}
    \caption{Poisson equation, test case TC2.
      Plots of the error curves versus the number of degrees of
      freedom $\Ndofs$ for the discretization using the (reduced)
      virtual element space of Section~\ref{S:lower} with $p_1=1$,
      $p_2=1,2,3$, $r=3$ on different polygonal mesh families and
      stabilization terms.
      The errors are measured using the energy norm (left) and the
      $\LTWO$-norm (right), and are expected to scale proportionally
      to $\hh^{3}$ and $\hh^4$, respectively.}
    \label{fig:poisson:C01-P2:hmax}
  \end{center}
\end{figure}
\subsection{Biharmonic equation}
\label{sec:numerics:Biharmonic}
In this section, we solve the two-dimensional biharmonic equation
using the conforming $\CS{1}$ virtual element approximation
corresponding to the parameters choice $p_2=p_1=r=2$.
We consider the two stabilization strategies for $\matU=\matI$ and
$\matU=\matD^{\perp}$ and the three possible choices of the parameter
$\alpha_{\mathsf{stab}}$ that are given by:
\begin{itemize}
\item $\alpha_{\mathsf{stab}}=\text{Trace}(\matM_{\P})/3$;
\item $\alpha_{\mathsf{stab}}=1\slash{\mP}$;
\item $\alpha_{\mathsf{stab}}=1\slash{\hh^2}$.
\end{itemize}

In Table \ref{tab:biharm_quad} and Figure~\ref{fig:err_biharm} (first
row from top) we report the computed condition number estimates and
the error curves for the family of quadrilateral meshes
(\textsf{QUAD}).
Note that $\mP=\hh^2$, so the stabilizations for these two
corresponding choices of $\alpha_{\mathsf{stab}}$ coincide and only
one set of results is shown.
All the stabilizations considered seem comparable in terms of the
condition number of the resulting linear system of equations,
exhibiting the expected growth $O(h^{-4})$.
Also they seem comparable in terms of accuracy, since the convergence
rate is the same (1 for the $H^2$ norm and 2 for the $L^\infty$ norm).
However, the stabilization with $\matU=\matI$ and
$\alpha_{\mathsf{stab}}=Trace(\matM_{\P})/3$ yields the best accuracy.

Table \ref{tab:biharm_tri} and Figure \ref{fig:err_biharm} (second
row) report the results obtained on triangular meshes (\textsf{TRI}).
Again all stabilizations considered exhibit the same behaviour in
terms of condition number growth and convergence rate, when refining
the mesh.
Differently from the case of quadrilateral meshes, the best accuracy
is obtained using the stabilization with $\matU=\matI$ and
$\alpha_{\mathsf{stab}}=1/h^2$.
Tables~\ref{tab:biharm_cvt}, \ref{tab:biharm_web} and Figure
\ref{fig:err_biharm} (third and fourth rows) report then the results
obtained on \textsf{CVT} and \textsf{HEX} meshes.
In these two cases, the performance of stabilization terms is
analogous to the one observed on quadrilateral meshes, since the best
accuracy is achieved by the stabilization with $\matU=\matI$ and
$\alpha_{\mathsf{stab}}=Trace(\matM_{\P})/3$.
%
\begin{figure}[!htb]
  \begin{center}
    \begin{tabular}{cc}
      \includegraphics[width=6cm]{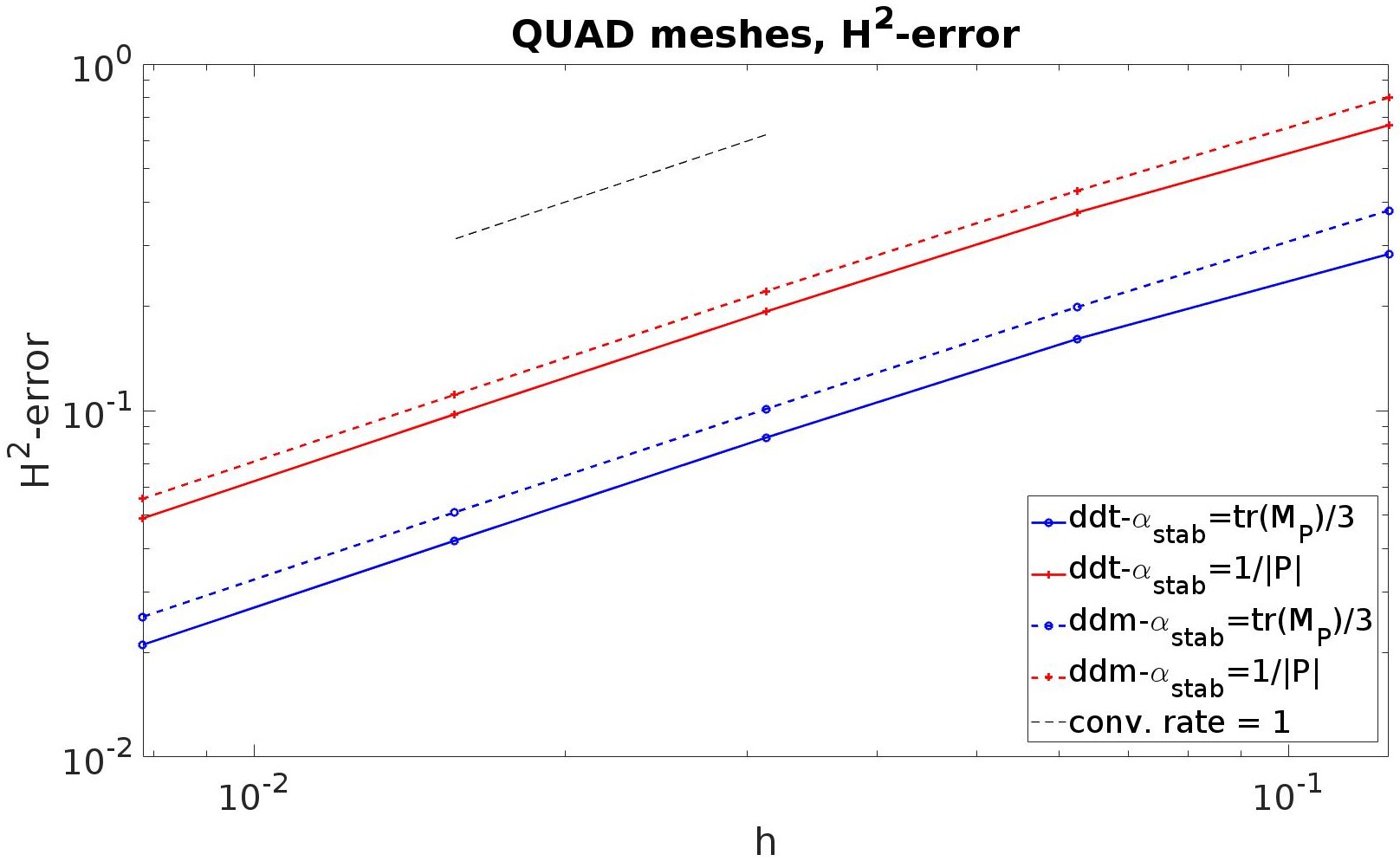}&   
      \includegraphics[width=6cm]{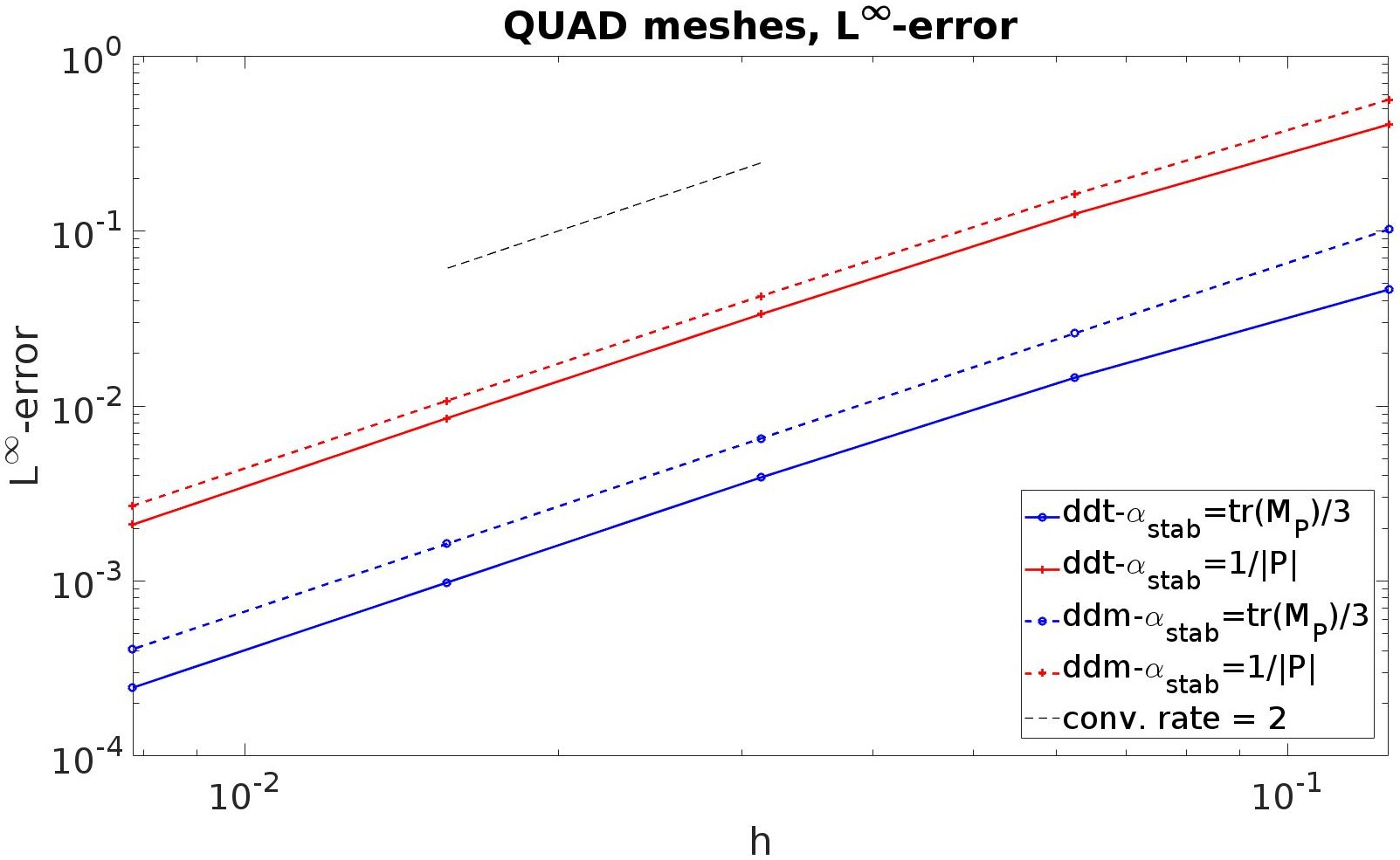}\\  
      \includegraphics[width=6cm]{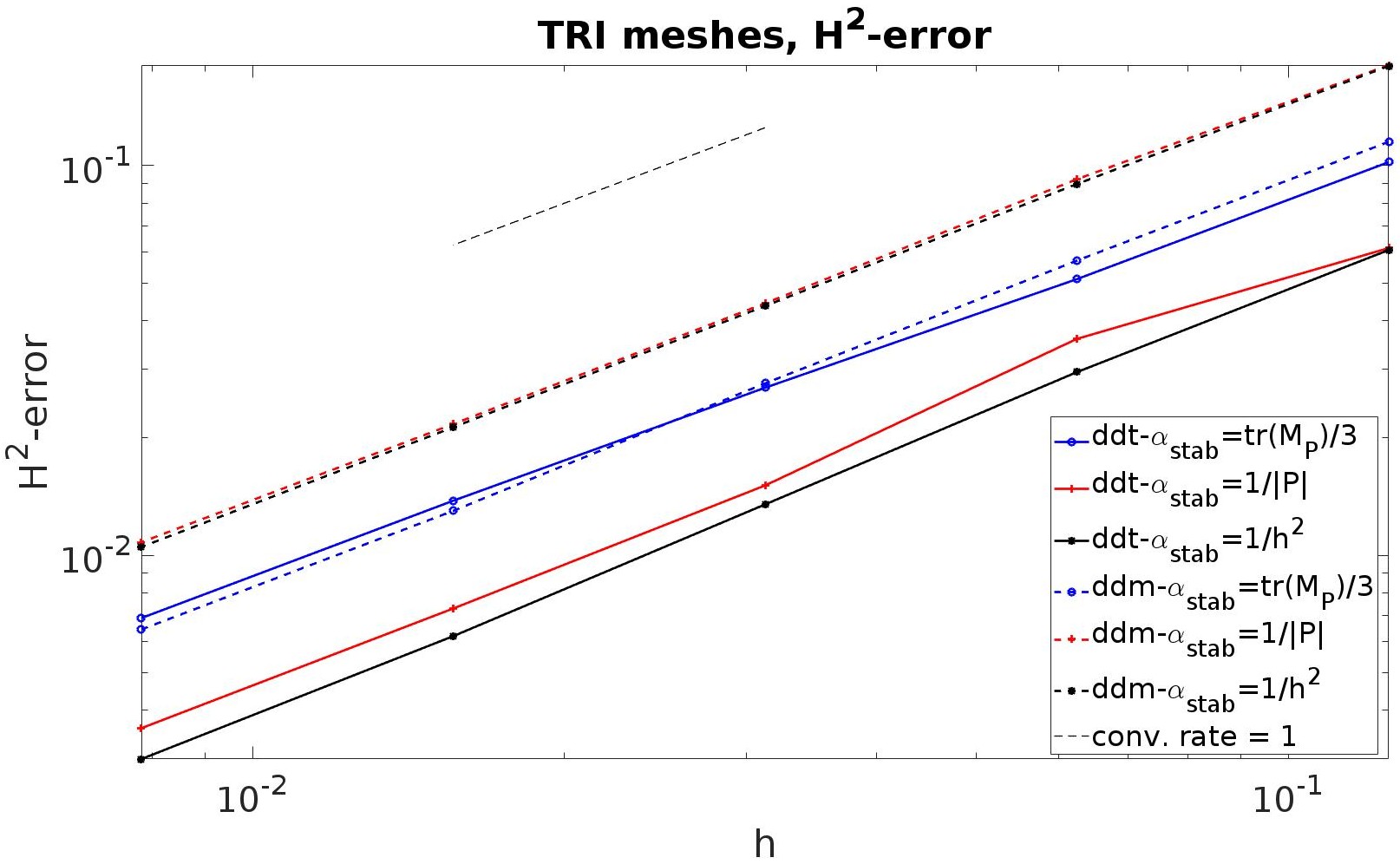}&   
      \includegraphics[width=6cm]{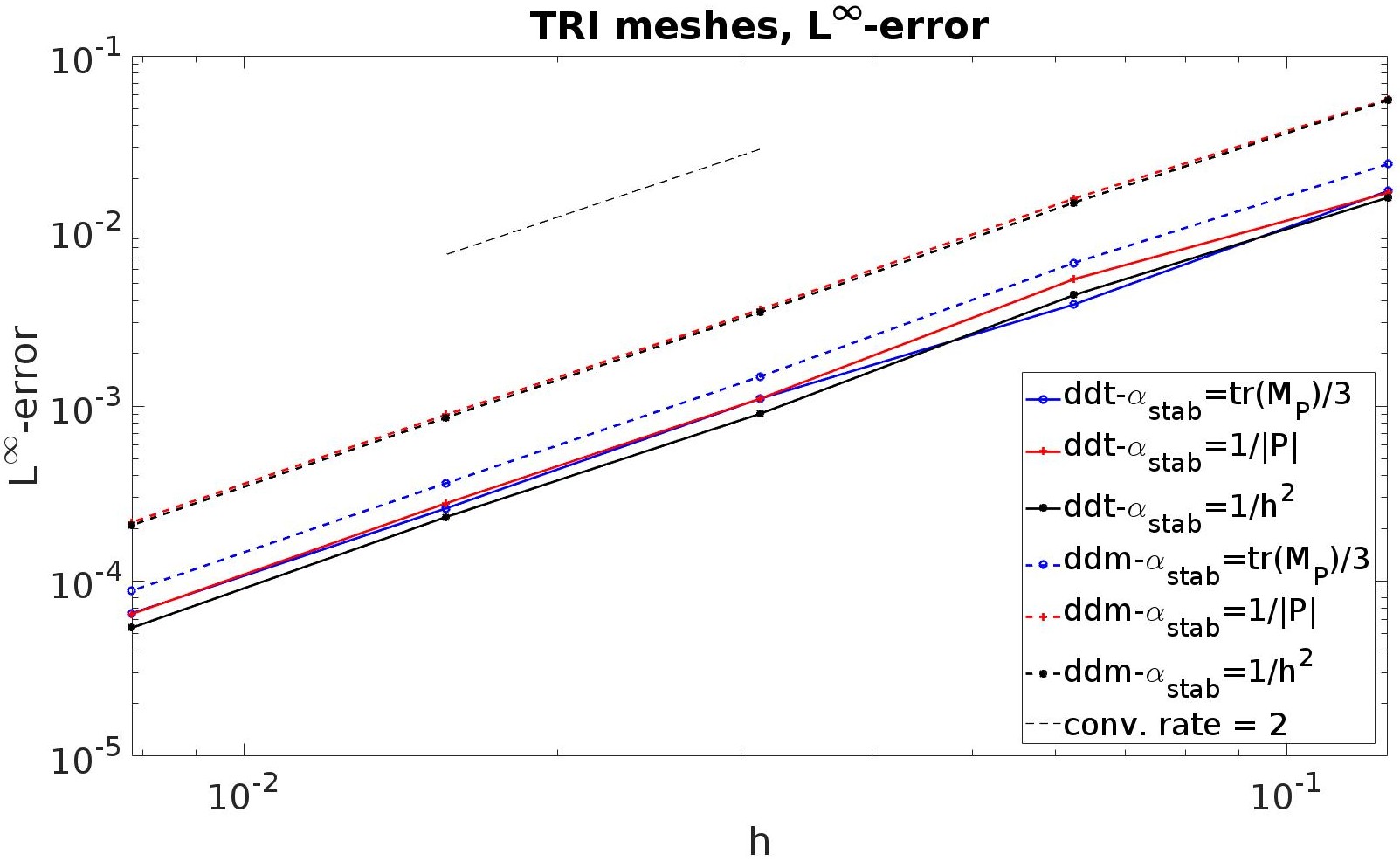}\\  
      \includegraphics[width=6cm]{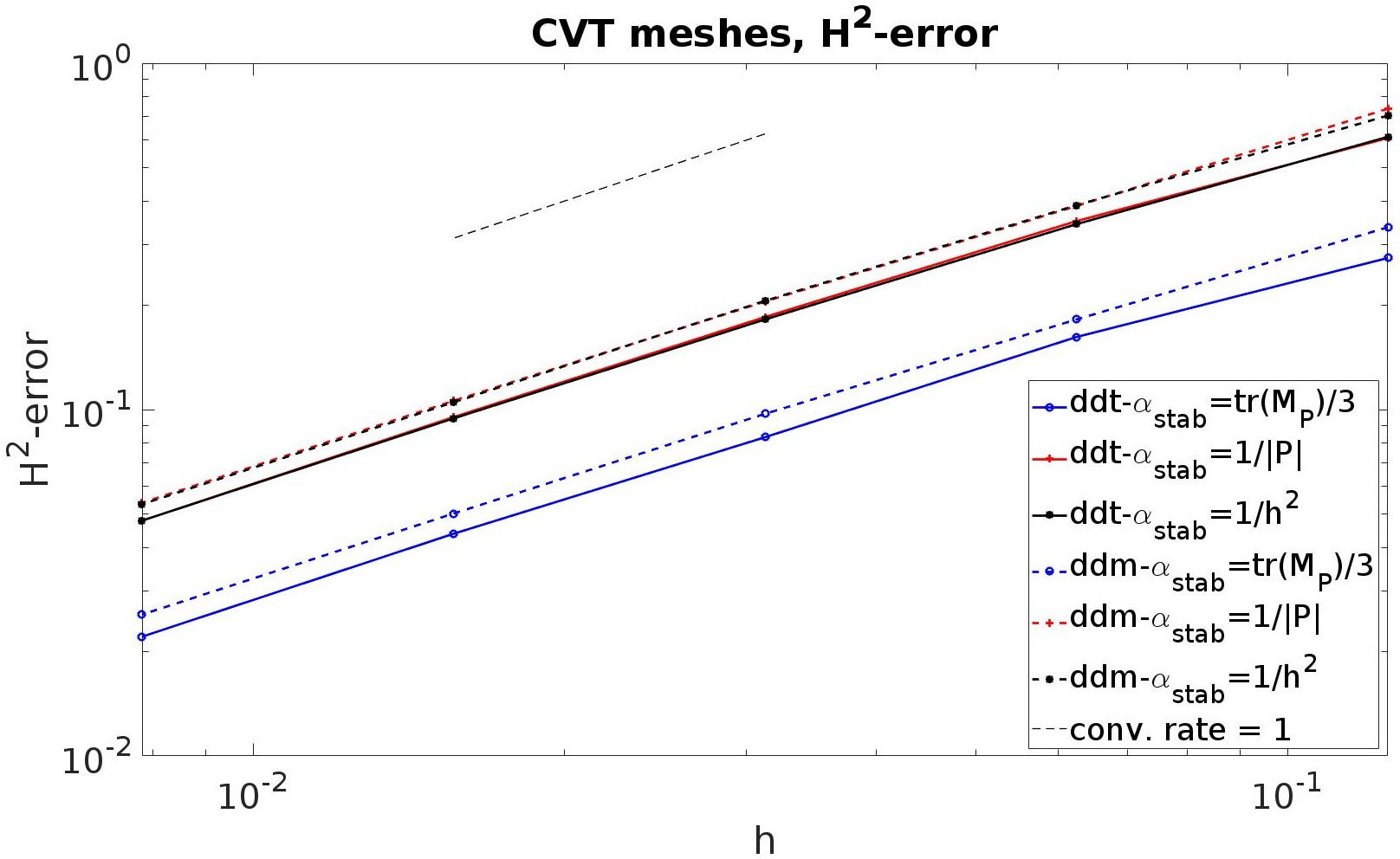}&   
      \includegraphics[width=6cm]{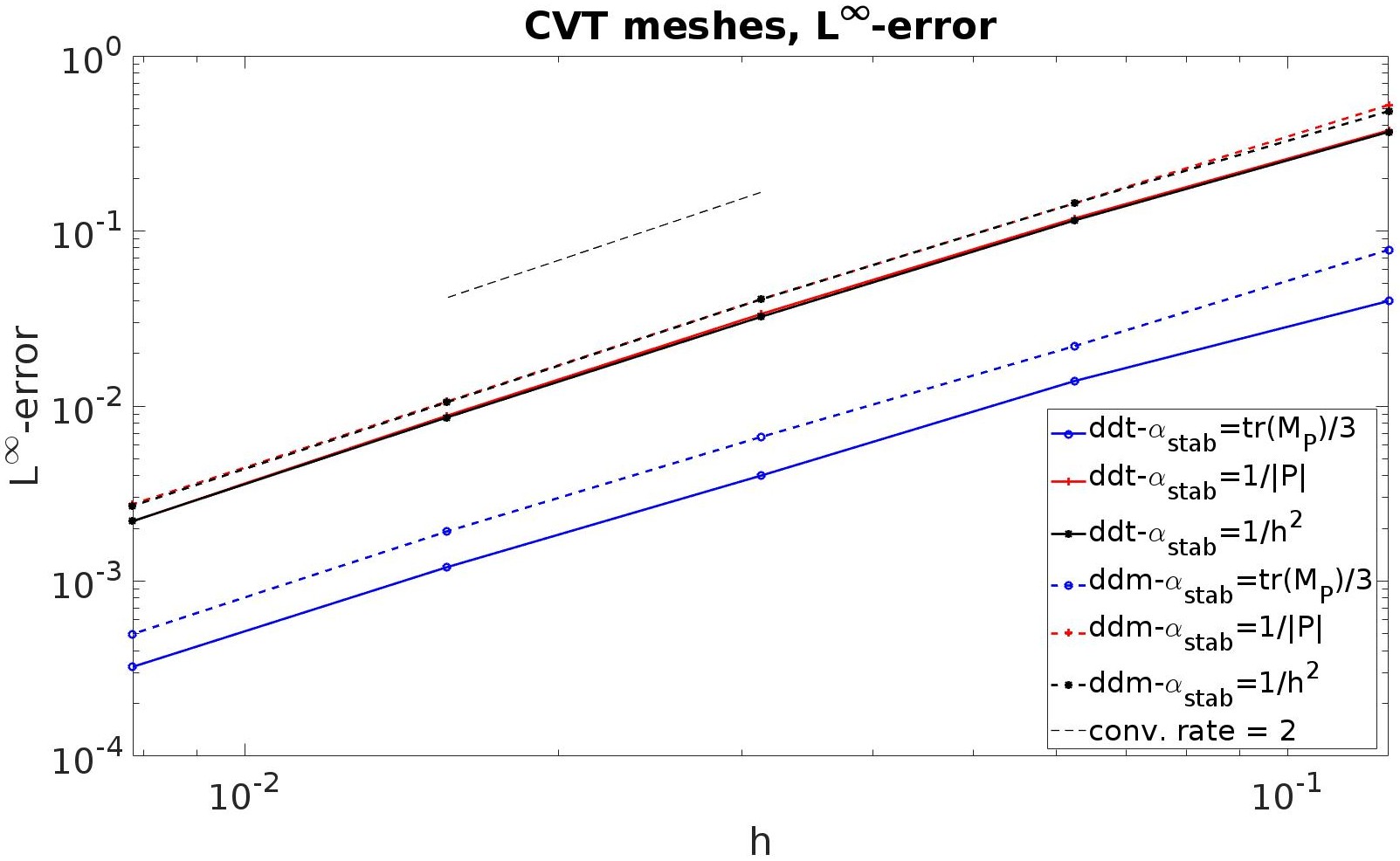}\\  
      \includegraphics[width=6cm]{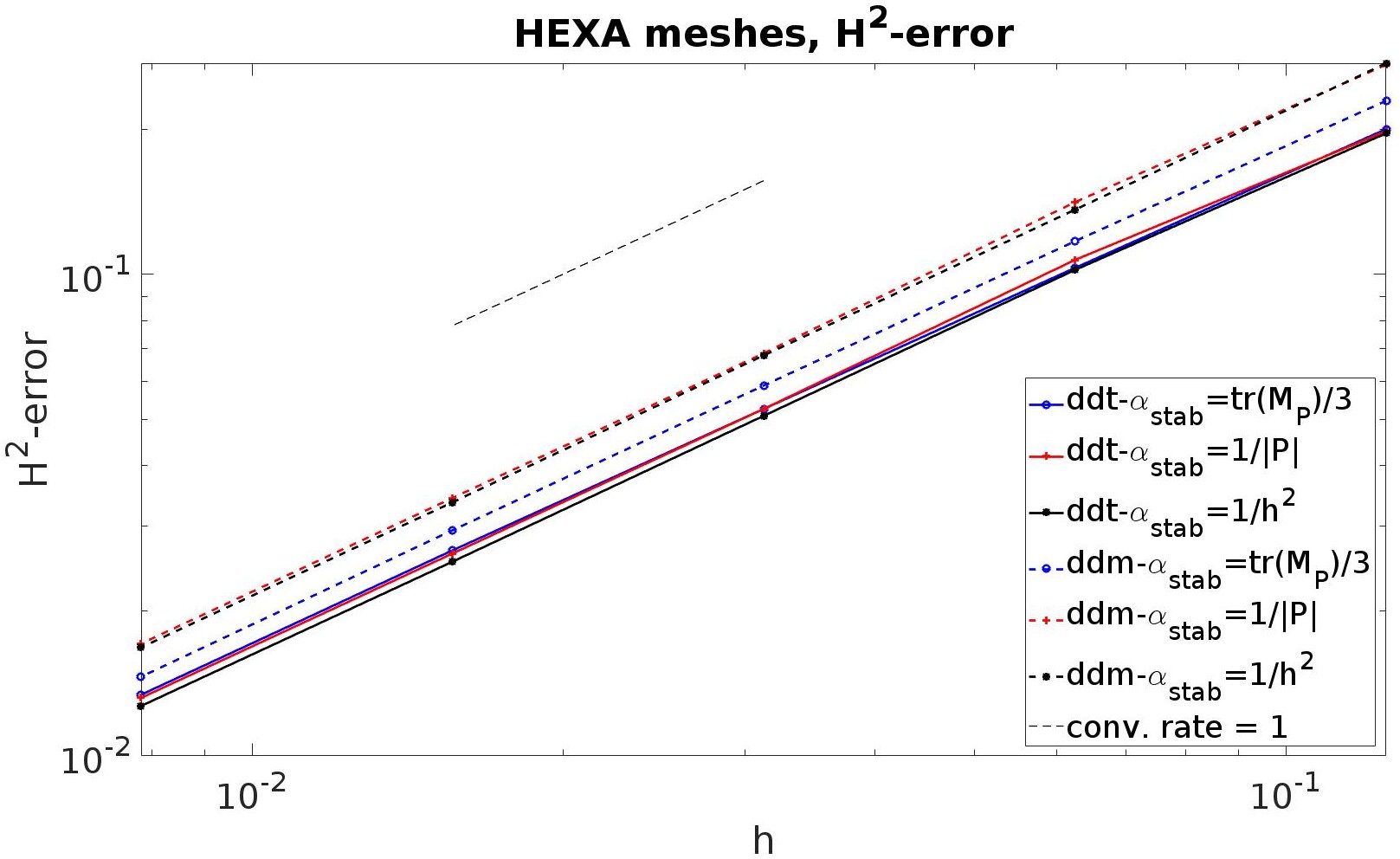}&   
      \includegraphics[width=6cm]{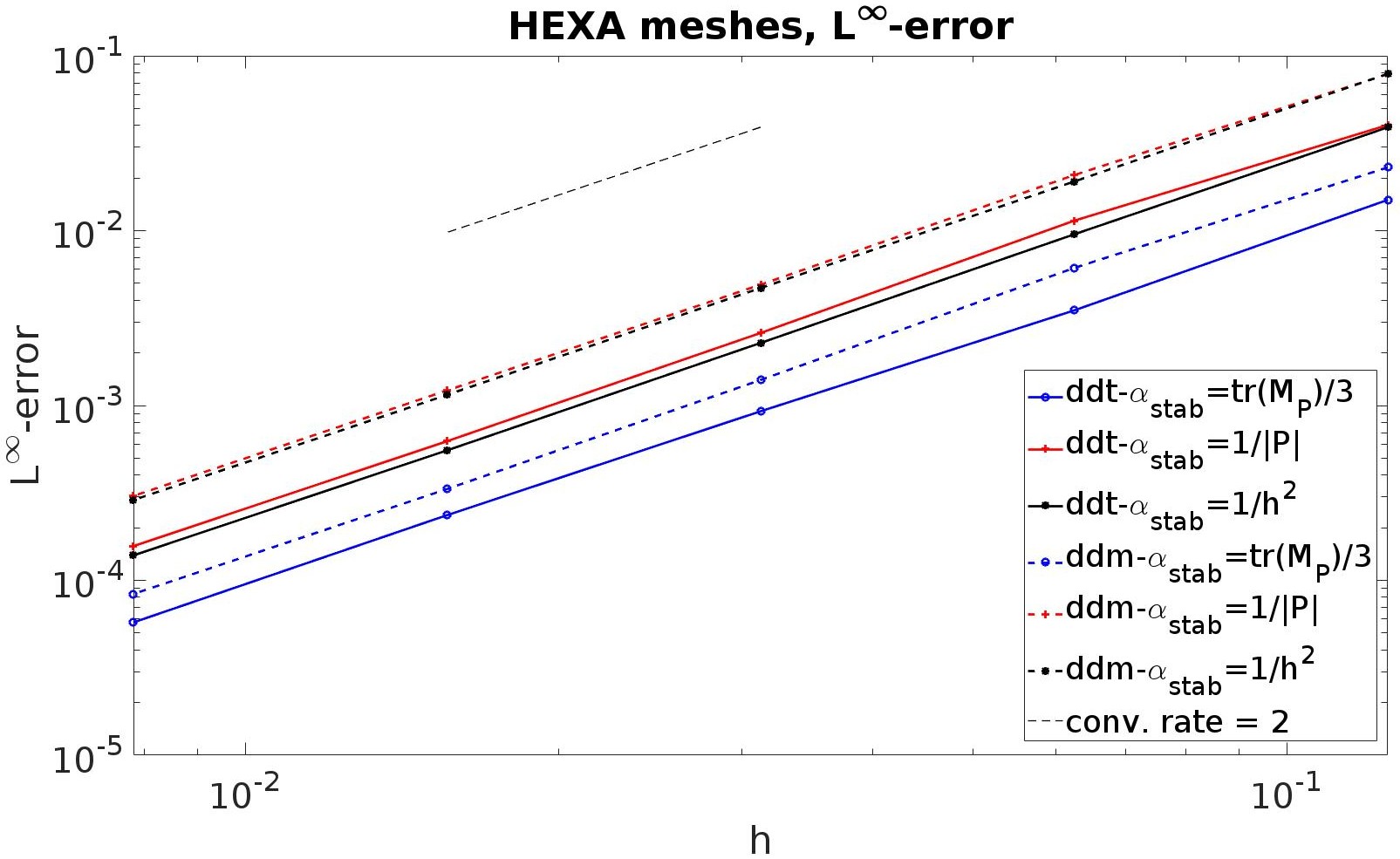}\\  
    \end{tabular}
    \caption{Biharmonic equation.  Plots of the error curves versus
      the mesh size $h$ on different polygonal mesh families and
      stabilization terms.
      The errors are measured using the $\HS{2}$-norm (left panels)
      and the $L^\infty$-norm (right panel), and are expected to scale
      proportionally to $\hh$ and $\hh^2$ , respectively.}
    \label{fig:err_biharm}
  \end{center}
\end{figure}


\renewcommand{\TABROW}[6]{ #1 & #2 & #3 & #4 & #5 & #6}
\begin{table}
  \begin{center}
    \def\arraystretch{1}\tabcolsep=5pt
    \begin{tabular}{r|r|rr|rr}
      \hline
      {}              &           &           \alphaOne &           \alphaTwo \\
      \TABROW{ $1/h$ }{ $\Ndofs$ }{  \stabI }{  \stabD }{  \stabI }{  \stabD }\\
      \hline
      \TABROW{     8 }{      243 }{ 5.77e+2 }{ 1.93e+2 }{ 3.41e+2 }{ 2.08e+3 }\\
      \TABROW{    16 }{      867 }{ 7.68e+3 }{ 2.47e+3 }{ 3.96e+3 }{ 2.28e+3 }\\
      \TABROW{    32 }{     3267 }{ 1.15e+5 }{ 3.65e+4 }{ 5.50e+4 }{ 3.09e+4 }\\
      \TABROW{    64 }{    12675 }{ 1.81e+6 }{ 5.72e+5 }{ 8.42e+5 }{ 4.69e+5 }\\
      \TABROW{   128 }{    49923 }{ 2.88e+7 }{ 9.11e+6 }{ 1.33e+7 }{ 7.40e+6 }\\
      \hline
    \end{tabular}
    \caption{Biharmonic equation, \textsf{QUAD} meshes. Comparison of
      the condition numbers obtained with the different stabilization
      strategies.}
    \label{tab:biharm_quad}
  \end{center}
\end{table}

\renewcommand{\alphaTwo}{\multicolumn{2}{c|}{$\alpha_{\mathsf{stab}}=1\slash{\mP}$}}
\renewcommand{\TABROW}[8]{ #1 & #2 & #3 & #4 & #5 & #6 & #7 & #8 }
\begin{table}
  \begin{center}
    \def\arraystretch{1}\tabcolsep=5pt
    \begin{tabular}{r|r|rr|rr|rr}
      \hline
      {}              &           &           \alphaOne &           \alphaTwo &         \alphaThree \\
      \TABROW{ $1/h$ }{ $\Ndofs$ }{  \stabI }{  \stabD }{  \stabI }{  \stabD }{  \stabI }{  \stabD }\\
      \hline
      \TABROW{     8 }{      369 }{ 2.34e+3 }{ 6.43e+2 }{ 1.46e+3 }{ 5.50e+2 }{ 9.68e+2 }{ 4.96e+2 }\\
      \TABROW{    16 }{     1404 }{ 5.27e+4 }{ 1.31e+4 }{ 3.16e+4 }{ 1.08e+4 }{ 1.90e+4 }{ 9.52e+3 }\\
      \TABROW{    32 }{     5424 }{ 9.46e+5 }{ 2.01e+5 }{ 5.96e+5 }{ 1.71e+5 }{ 3.10e+5 }{ 1.48e+5 }\\
      \TABROW{    64 }{    21507 }{ 1.39e+7 }{ 3.35e+6 }{ 8.76e+6 }{ 2.75e+6 }{ 4.92e+6 }{ 2.40e+6 }\\
      \TABROW{   128 }{    86169 }{ 3.30e+8 }{ 6.42e+7 }{ 2.17e+8 }{ 5.53e+7 }{ 9.97e+7 }{ 4.70e+7 }\\
      \hline
    \end{tabular}
    \caption{Biharmonic equation, \textsf{TRI} meshes. Comparison of the
      condition numbers obtained with the different stabilization
      strategies.}
    \label{tab:biharm_tri}
  \end{center}
\end{table}

\begin{table}
  \begin{center}
    \def\arraystretch{1}\tabcolsep=5pt
    \begin{tabular}{r|r|rr|rr|rr}
      \hline
      {}              &           &           \alphaOne &           \alphaTwo &         \alphaThree \\
      \TABROW{ $1/h$ }{ $\Ndofs$ }{  \stabI }{  \stabD }{  \stabI }{  \stabD }{  \stabI }{  \stabD }\\
      \hline
      \TABROW{     8 }{      474 }{ 6.42e+2 }{ 1.95e+2 }{ 3.22e+2 }{ 1.87e+2 }{ 3.40e+2 }{ 1.75e+2 }\\
      \TABROW{    16 }{     1704 }{ 8.66e+3 }{ 2.82e+3 }{ 4.32e+3 }{ 2.09e+3 }{ 3.97e+3 }{ 2.11e+3 }\\
      \TABROW{    32 }{     6438 }{ 1.55e+5 }{ 4.14e+4 }{ 6.77e+4 }{ 2.99e+4 }{ 6.52e+4 }{ 3.05e+4 }\\
      \TABROW{    64 }{    24921 }{ 2.90e+6 }{ 6.70e+5 }{ 1.06e+6 }{ 4.98e+5 }{ 9.74e+5 }{ 4.52e+5 }\\
      \TABROW{   128 }{    98724 }{ 4.53e+7 }{ 1.14e+7 }{ 1.99e+7 }{ 4.98e+5 }{ 1.56e+7 }{ 8.00e+6 }\\
      \hline
    \end{tabular}
    \caption{Biharmonic equation, \textsf{CVT} meshes. Comparison of
      the condition numbers obtained with the different stabilization
      strategies.}
    \label{tab:biharm_cvt}
  \end{center}
\end{table}

\begin{table}
  \begin{center}
    \def\arraystretch{1}\tabcolsep=5pt
    \begin{tabular}{r|r|rr|rr|rr}
      \hline
      {}              &           &           \alphaOne &           \alphaTwo &         \alphaThree \\
      \TABROW{ $1/h$ }{ $\Ndofs$ }{  \stabI }{  \stabD }{  \stabI }{  \stabD }{  \stabI }{  \stabD }\\
      \hline
      \TABROW{     8 }{     1371 }{ 1.33e+4 }{ 2.68e+3 }{ 6.70e+3 }{ 2.20e+3 }{ 4.51e+3 }{ 1.99e+3 }\\
      \TABROW{    16 }{     5415 }{ 3.52e+5 }{ 5.57e+4 }{ 1.94e+5 }{ 4.50e+4 }{ 8.87e+4 }{ 3.80e+4 }\\
      \TABROW{    32 }{    21303 }{ 5.46e+6 }{ 8.81e+5 }{ 2.82e+6 }{ 7.07e+5 }{ 1.35e+6 }{ 5.89e+5 }\\
      \TABROW{    64 }{    85251 }{ 7.79e+7 }{ 1.44e+7 }{ 4.30e+7 }{ 1.09e+7 }{ 2.52e+7 }{ 9.58e+6 }\\
      \TABROW{   128 }{   343131 }{ 1.04e+9 }{ 2.62e+8 }{ 8.72e+8 }{ 2.20e+8 }{ 4.26e+8 }{ 1.83e+8 }\\
      \hline
    \end{tabular}
    \caption{Biharmonic equation, \textsf{HEX} meshes. Comparison of
      the condition numbers obtained with the different stabilization
      strategies.}
    \label{tab:biharm_web}
  \end{center}
\end{table}

\section{Conclusion}
\label{sec6:conclusions}

We reviewed the construction of highly regular virtual element
approximations for polyharmonic problems in two spatial dimensions,
recalling the main theoretical convergence results available in the
literature.
Moreover, we performed a set of new two-dimensional numerical tests to
investigate how different stabilizations in the formulation of the VEM
affect the solver performance in terms of condition number of the
resulting linear system and accuracy of the approximation schemes.
For the discretization of the Poisson equation, our numerical results
show that the choice of the stabilization has an almost negligible
effect on condition numbers and accuracy.
On the other hand, the numerical results that we obtained for the
biharmonic equation shows that the choice of the stabilization may
affect significantly the accuracy of approximation. This effect may be
even more pronounced for $p_1>2$ and requires further investigation.
The best overall performance in our tests is provided by the so-called
\emph{dofi-dofi} stabilization. On the basis of the results obtained
regarding the conditioning of the highly regular VEM matrices, we also
believe that it is worth of future investigations the development of
effective preconditioners for VEM approximations of high order
elliptic equations.



\section*{Acknowledgements}
PFA and MV acknowledge the financial support of PRIN research grant
number 201744KLJL ``\emph{Virtual Element Methods: Analysis and
Applications}'' funded by MIUR.  PFA and MV acknowledge the
financial support of INdAM-GNCS.
GM acknowledges the financial support of the ERC Project CHANGE,
which has received funding from the European Research Council under
the European Union’s Horizon 2020 research and innovation program
(grant agreement no.~694515).


\clearpage

\end{document}